\documentclass[12pt]{article}
\usepackage{mathrsfs}
\usepackage{amssymb}
\usepackage{listings}
\usepackage{amsmath}
\usepackage{pict2e}
\usepackage{amsfonts}
\usepackage{graphicx}
\usepackage[small]{caption}
\usepackage{subfigure}
\usepackage{cite}
\usepackage{mathtools}
\usepackage{multirow}

\def\0{\emptyset}

\begin{document}
\newtheorem{claim}{Claim}[section]
\newtheorem{theorem}{Theorem}[section]
\newtheorem{corollary}[theorem]{Corollary}
\newtheorem{definition}[theorem]{Definition}
\newtheorem{conjecture}[theorem]{Conjecture}
\newtheorem{question}[theorem]{Question}
\newtheorem{lemma}[theorem]{Lemma}
\newtheorem{algorithm}[theorem]{Algorithm}
\newtheorem{assump}[theorem]{Assumption}
\newtheorem{proposition}[theorem]{Proposition}
\newtheorem{Observation}[theorem]{Observation}
\newtheorem{Case}[theorem]{Case}
\newenvironment{proof}{\noindent {\bf
Proof.}}{\rule{2mm}{2mm}\par\medskip}
\newcommand{\remark}{\medskip\par\noindent {\bf Remark.~~}}
\newcommand{\pp}{{\it p.}}
\newcommand{\de}{\em}

\title{\bf Anti-Ramsey problems in the generalized Petersen graphs for cycles}
\author{{{\small\bf Huiqing LIU$^1$}\thanks{Partially supported by NNSFC (No.
11971158), email: hql\_2008@163.com; }\quad{\small\bf Mei
LU$^2$}\thanks{Partially supported by NNSFC (No. 12171272), email:
lumei@mail.tsinghua.edu.cn; }\quad{\small\bf Shunzhe ZHANG$^1$}\thanks{Partially
supported by Youth Project Funds of Hubei Provincial Department of Education, email: shunzhezhang@hubu.edu.cn.}
}\\ {\small $^1$Hubei Key Laboratory of Applied Mathematics, Faculty of
Mathematics and Statistics, }\\{\small Hubei University, Wuhan 430062, China}\\
{\small $^2$Department of Mathematical Sciences, Tsinghua University, Beijing
100084, China.}}
\date{}

\maketitle \baselineskip 17.0pt

\begin{abstract}

The anti-Ramsey number $Ar(G, H)$ is the maximum number of colors in an
edge-coloring of $G$ with no rainbow copy of $H$. In this paper, we determine
the exact anti-Ramsey number in the generalized Petersen graph $P_{n,k}$ for
cycles $C_d$, where $1\leq k\leq \lfloor \frac{n-1}{2} \rfloor$ and $5\le d \le
6$. We also give an algorithm to obtain the upper bound or lower bound of anti-Ramsey number.

\vskip 0.1cm

\end{abstract}

\section{Introduction}

Let $G=(V(G), E(G))$ be a graph. We denote by $E(v)$ the set of all edges
incident with $v$ in $G$, and by $C_d$ a cycle of order $d$. Let $M(G)$ be the
maximum number of vertex-disjoint cycles with length at least $2$ in $G$.

An edge-coloring of a graph $G$ is a mapping $\phi: E(G)\rightarrow
\{1,2,\ldots,c\}$, where $\{1,2,\ldots,c\}$ is a set of colors.
An edge colored graph is called {\em rainbow} if all the colors on the edges are
distinct. A {\em rainbow copy} of a graph $H$ in an edge-colored graph $G$ is a
subgraph of $G$ isomorphic to $H$ such that the coloring restricted to $H$ is
rainbow. Given two graphs $G$ and $H$, the {\em anti-Ramsey number} $Ar(G, H)$ is
the maximum number of colors in an edge-coloring of $G$ which has no rainbow copy
of $H$. If there is no copy of $H$ in $G$, then $Ar(G, H)=|E(G)|$.

The study of anti-Ramsey theory was initiated by Erd$\ddot{\mbox{o}}$s,
Simonovits and S$\acute{\mbox{o}}$s in 1973 \cite{Erd01} and considered in the
classical case when $G=K_n$. They presented a close relationship between the
anti-Ramsey number and Tur$\acute{\mbox{a}}$n number. Since then, plentiful
results were researched for a variety of graphs $H$, including cycles
\cite{Alon83,Axen04,Mon05,Jiang03,Jin09,Xu}, cliques \cite{Erd01,Mon02}, trees
\cite{Jahan16,Jia19}, and matchings \cite{Haas12,Li09}. Some other graphs were
also considered as the host graphs in anti-Ramsey problems, such as hypergraphs
\cite{Gu20}, hypecubes \cite{Axen07}, complete split graphs
\cite{Gorgol16,Jin18}, and triangulations \cite{Hor12,Jendrol14,Lan19}.

For cycles, Erd$\ddot{\mbox{o}}$s, Simonovits and S$\acute{\mbox{o}}$s
\cite{Erd01} showed that $Ar(K_n, C_3)=n-1$. Alon \cite{Alon83} showed that
$Ar(K_n,C_4)=\lfloor \frac{4n}{3} \rfloor -1$. Montellano-Ballesteros and
Neumann-Lara \cite{Mon05} proved that $Ar(K_n,
C_d)=(\frac{d-2}{2}+\frac{1}{d-1})n+O(1)$ for $3\leq d\leq n$. In 2015,
Hor$\check{\mbox{n}}$$\acute{\mbox{a}}$k, Jendrol, Schiermeyer and
Sot$\acute{\mbox{a}}$k \cite{Hor12} investigated the anti-Ramsey numbers for
cycles in plane triangulations. Lan, Shi and Song \cite{Lan19} gave the upper and
lower bounds of the anti-Ramsey number $Ar(W_n, C_d)$ for $n\geq d-1$ and $d\geq
6$. Recently, Xu, Lu and Liu \cite{Xu} completely solved the anti-Ramsey numbers
of cycles in wheel graphs and determined some exact values of $Ar(G, C_d)$ when
the host graph $G$ is $P_m\times P_n$, $P_m\times C_n$, $C_m\times C_n$ and
cyclic Cayley graph, respectively.

Here, we are interested in the anti-Ramsey problems in generalized Petersen
graphs for cycles. Let $n\geq 3$ and $k$ be integers such that $1\leq k\leq n-1$.
The generalized Petersen graph $P_{n,k}$ is a graph of order $2n$ with
$V(P_{n,k})=\{u_i,v_i : 0\leq i\leq n-1\}$ and edge set
$E(P_{n,k})=\{u_iu_{i+1},u_iv_i,v_iv_{i+k} : 0\leq i\leq n-1\}$, where subscripts
are taken modulo $n$. Let $O_n=\{u_0,u_1,\ldots,u_{n-1}\}$ and
$I_n=\{v_0,v_1,\ldots,v_{n-1}\}$. Then $V(P_{n,k})=O_n\cup I_n$. The subgraph
induced by $O_n$ is called the outer rim, while the subgraph induced by $I_n$ is
called the inner rim. A spoke of $P_{n,k}$ is an edge of the form $u_iv_i$ for
some $0\leq i\leq n-1$. Note that $P_{n,k}$ is isomorphic to $P_{n,n-k}$. Hence
we can always assume that $1\leq k\leq \lfloor \frac{n-1}{2} \rfloor$.

The generalized Petersen graph has been extensively investigated as many nice
structural and algorithmic properties.

\begin{proposition}
\label{1.1}
$(i)$ $P_{n,k}$ contains $C_5$ if and only if $n\in \{3,5\}$ or $k\in
\{2,\frac{n}{5},\frac{2n}{5},\frac{n-1}{2}\}$;

$(ii)$ $P_{n,k}$ contains $C_6$ if and only if $k\in
\{1,3,\frac{n}{6},\frac{n-1}{3},\frac{n+1}{3},\frac{n-2}{2}\}$.
\end{proposition}

\begin{proof} Note that $V(P_{n,k})=O_n\cup I_n$. By the definition of $P_{n,k}$,
$|V(C_5)\cap O_n|\in \{0,2,3,5\}$ and $|V(C_6)\cap O_n|\in \{0,2,3,4,6\}$.
Obviously, $P_{n,1}$ contains $C_5$ for $n\in \{3,5\}$ and $C_6$ for $n\geq 3$.
So, in the following, we can assume $k\geq 2$.

Let $C_t=w_1w_2\cdots w_{t}w_1$ be a cycle of length $t$ in $P_{n,k}$.  If
$V(C_t)\subseteq O_n$, then $t=n$. If $|V(C_t)\cap O_n|=0$, i.e.,
$V(C_t)\subseteq I_n$, then we can assume $w_1=v_0$ by symmetry, and then
$C_t=v_0v_{k}v_{2k}\cdots v_{(t-1)k}v_0$, which implies $tk\equiv 0~(\mbox{mod~}
n)$ and $ik\not\equiv 0~(\mbox{mod~} n)$ for $1\le i\le t-1$. Combining this with
$1\leq k\leq \lfloor \frac{n-1}{2} \rfloor$, we have $k\in
\{\frac{n}{5},\frac{2n}{5}\}$ for $t=5$ and $k=\frac{n}{6}$ for $t=6$. Thus we
may assume $2\le |V(C_t)\cap O_n|\le t-1$. Without loss of generality, we assume
$w_1=u_0$ and $w_2=u_1$.

(i)  Note that $|V(C_5)\cap O_n|\in \{2,3\}$. If $|V(C_5)\cap O_n|=2$, then
$|V(C_5)\cap I_n|=3$,  $w_5=v_0$ and $w_3=v_1$. Let $w_4=v_s$, then
$v_1v_s,v_sv_0\in E(P_{n,k})$, which implies $s-1=k$ and $n-s=k$. So
$k=\frac{n-1}{2}$.

If $|V(C_5)\cap O_n|=3$, that is $|V(C_5)\cap I_n|=2$, then we can assume that
$w_3=u_{2}$, and $w_4=v_2,w_5=v_0$, which implies $v_2v_0\in E(P_{n,k})$. Thus
$k=2$.

(ii)  Note that $|V(C_6)\cap O_n|\in \{2,3,4\}$ and $N_{I_n}(v_i)=\{v_{i+k},
v_{n-k+i}\}$ for $0\le i\le n-1$.

 If $|V(C_6)\cap O_n|=2$, then $w_3=v_1$, $w_6=v_0$, $w_4=v_{k+1}$ (resp.
 $w_4=v_{n-k+1}$) and $w_5=v_{2k+1}$ (resp. $w_5=v_{n-2k+1}$), which implies
 $3k+1\equiv 0~(\mbox{mod~} n)$ (resp. $n-3k+1\equiv 0~(\mbox{mod~} n)$), i.e.,
 $k=\frac{n-1}{3}$ (resp. $k=\frac{n+1}{3}$).

 If $|V(C_6)\cap O_n|=3$, say $w_3=u_2$, then  $w_4=v_2$, and $w_5=v_{k+2}$
 (resp. $w_5=v_{n-k+2}$), which implies $2k+2\equiv 0~(\mbox{mod~} n)$ (resp.
 $n-2k+2\equiv 0~(\mbox{mod~} n)$), i.e.,  $k=\frac{n-2}{2}$.

 If $|V(C_6)\cap O_n|=4$, then we can assume that $w_3=u_2$ and $w_4=u_3$. Thus
 $w_5=v_{3}$, which implies $k=3$ as $w_6=v_0$.
\end{proof}

In this paper, we determine the exact anti-Ramsey number in $P_{n,k}$ for $C_d$,
where $1\leq k\leq \lfloor \frac{n-1}{2} \rfloor$ and $5\le d \le 6$. Our results
are as follows.

\begin{theorem}
\label{1.2}
For $1\leq k\leq \lfloor \frac{n-1}{2} \rfloor$, we have
\begin{eqnarray}
  Ar(P_{n,k},C_5)=\left\{
\begin{array}{ll}
   7,                                     &\mbox{if}~(n,k)=(3,1);
   \\          \nonumber
   13,                                    &\mbox{if}~(n,k)=(5,1);
   \\          \nonumber
   10,                                    &\mbox{if}~(n,k)=(5,2);
   \\          \nonumber
   22,                                    &\mbox{if}~(n,k)=(10,2);
   \\          \nonumber
   \lfloor \frac{7n}{3} \rfloor,          &\mbox{if}~n\geq 6, n\neq
   10~\mbox{and}~k=2,~\mbox{or}~k=\frac{n-1}{2}\ge 3;   \\    \nonumber
   \frac{14n}{5},                         &\mbox{if}~k\geq
   3~\mbox{and}~k\in\{\frac{n}{5},\frac{2n}{5}\};                     \\
   \nonumber
   \,3n,                                     &otherwise.
   \nonumber
   \end{array}\right.
\end{eqnarray}
\end{theorem}


\begin{theorem}
\label{1.3}
For $n\ge 3$, we have
\begin{eqnarray}
  Ar(P_{n,1},C_6)=\left\{
\begin{array}{ll}
   7,                                     &\mbox{if}~n=3;
   \\          \nonumber
   9,                                     &\mbox{if}~n=4;
   \\          \nonumber
  14,                                    &\mbox{if}~n=6;
  \\          \nonumber
  \lfloor \frac{5n}{2} \rfloor,  &\mbox{if}~n\ge 5,~n\neq 6;
     \nonumber
   \end{array}\right.
\end{eqnarray}
\end{theorem}


\begin{theorem}
\label{1.4}
For $n\ge 5$, we have
\begin{eqnarray}
  Ar(P_{n,2},C_6)=\left\{
\begin{array}{ll}
   11,                                    &\mbox{if}~n=5;
   \\          \nonumber
   14,                                    &\mbox{if}~n=6;
   \\          \nonumber
   17,                                    &\mbox{if}~n=7;
   \\          \nonumber
   34,                                    &\mbox{if}~n=12; \\          \nonumber
   3n,          &otherwise.
   \nonumber
   \end{array}\right.
\end{eqnarray}
\end{theorem}


\begin{theorem}
\label{1.5}
For $n\ge 7$, we have
\begin{eqnarray}
  Ar(P_{n,3},C_6)=\left\{
\begin{array}{ll}
  17,                                    &\mbox{if}~n=8;
  \\          \nonumber
   22,                                    &\mbox{if}~n=10;
   \\          \nonumber
   42,                                    &\mbox{if}~n=18;
   \\          \nonumber
   \lfloor \frac{5n}{2} \rfloor,          &otherwise.
   \nonumber
   \end{array}\right.
\end{eqnarray}
\end{theorem}


\begin{theorem}
\label{1.6}
For $4\leq k\leq \lfloor \frac{n-1}{2} \rfloor$, we have
\begin{eqnarray}
  Ar(P_{n,k},C_6)=\left\{
\begin{array}{ll}
   \lfloor \frac{5n}{2} \rfloor,          &\mbox{if}~ k\in \{\frac{n-1}{3},
   \frac{n+1}{3}\};                                       \\          \nonumber
   \lfloor \frac{7n}{3} \rfloor,          &\mbox{if}~k=\frac{n-2}{2};
   \\          \nonumber
   \frac{17n}{6},                         &\mbox{if}~k=\frac{n}{6};
   \\          \nonumber
   \,3n,                                     &otherwise.
   \nonumber
   \end{array}\right.
\end{eqnarray}
\end{theorem}

The rest of the paper is organized as follows. In Section $2$, we give some
notations and lemmas which will be used in the proof of main results. The exact values of $Ar(P_{n,k}, C_5)$ and $Ar(P_{n,k},
C_6)$ for $1\leq k\leq \lfloor \frac{n-1}{2} \rfloor$ are determined in Sections
$3$ and 4, respectively.

\section{Preliminaries}

A hypergraph is a pair $H=(V, \mathscr{F})$, where $V$ is a finite set of
vertices and $\mathscr{F}$ is a family of subsets of $V$ such that for every
$F\in \mathscr{F}$, $F\neq \emptyset$ and $V=\bigcup_{F\in \mathscr{F}}F$. If
$|F|=1$, we call $F$ a loop. The rank of $H$ is defined as $r(H)=\max_{F\in
\mathscr{F}}|F|$. If $r(H) = 2$, then $H$ is a graph. For every $v\in V$, denote
$F_H(v)=\{F\in \mathscr{F}|v\in F\}$. Let $s(H)=\max\{|F_i\cap F_j|:F_i,F_j\in
\mathscr{F}\}$.

Let $G$ be a simple graph with $E(G) = \{e_1,e_2,\ldots,e_m\}$, and let
$\Psi=\{G_1,\ldots,G_l\}$ with $E(G)=\cup^l_{i=1}E(G_i)$, where $G_i$ is a
subgraph of $G$ for $1\leq i\leq l$. A hypergraph $H_{G,\Psi}$ is constructed in
\cite{Xu} with $V(H_{G,\Psi})=\{x_1,x_2,\ldots,x_l\}$ and
$E(H_{G,\Psi})=\{F_1,F_2,\ldots,F_m\}$, where each $x_i$ corresponds to $G_i$ in
$\Psi$ ($1\le i\le  l$) and $x_i\in F_j$ if and only if $e_j\in E(G_i)$ ($1\le
j\le  m$).

Given a graph $G$ and $\Psi=\{G_1,\ldots,G_l\}$, where $G_i$ is a subgraph of $G$
for $1\leq i\leq l$. For an edge-coloring of $G$, $H\in \Psi$ is called {\em
rainbow-$\Psi$ graph} if the coloring restricted to $H$ is rainbow. Let
$Ar_G(\Psi)$ denote the maximum number of colors in a coloring of the edges of
$G$ such that there is no rainbow-$\Psi$ graph of $G$. Obviously, if
$\Psi=\{G_1,\ldots,G_l\}$ is the set of all copies of $H$ in $G$, then $Ar(G,
H)=Ar_G(\Psi)$.

The following two lemmas will be useful in the proof of our main results.

\begin{lemma} \cite{Xu}
\label{2.1}
Let $G=(V,E)$ be a simple graph. Let $\Psi=\{G_1,\ldots,G_l\}$ such that
$E(G)=\cup^l_{i=1}E(G_i)$, where $G_i$ is a subgraph of $G$ for $1\leq i\leq l$.
Suppose $H_{G,\Psi}$ is a graph. Then
\begin{align}
Ar_G(\Psi)=|E(G)|-|\Psi|+M(H_{G,\Psi}).\nonumber
\end{align}
\end{lemma}

\begin{lemma} \cite{Xu}
\label{2.2}
Let $G=(V,E)$ be a simple graph. Let $\Psi=\{G_1,\ldots,G_l\}$ such that
$E(G)=\cup^l_{i=1}E(G_i)$, where $G_i$ is a subgraph of $G$ for $1\leq i\leq l$.
Suppose $H_{G,\Psi}$ is a hypergraph with $r(H_{G,\Psi})=r$ and
$s(H_{G,\Psi})=s$, then
\begin{align}
Ar_G(\Psi)\leq |E(G)|-\left\lceil \frac{2l}{r+s} \right \rceil.     \nonumber
\end{align}
\end{lemma}

Let $H=(V,\mathscr{F})$ be a hypergraph with $|V(H)|=l$ and
$\mathscr{F}=\{F_1,F_2,\ldots,F_m\}$. Given an integer $h$ with $0\leq h\leq
m-1$. Let $\{\mathscr{F}_1,\mathscr{F}_2,\ldots,\mathscr{F}_{m-h}\}$ be an
$(m-h)$-partition of $\mathscr{F}$ with
$|\mathscr{F}_1|\geq |\mathscr{F}_{2}|\geq \cdots \geq |\mathscr{F}_{m-h}|\ge 1$.
For $1\leq i\leq m-h$, we denote
$
  L_i:=
   \{v\in F_t\cap F_j~|~F_t,F_j\in \mathscr{F}_i,~t\neq
   j\}~\mbox{if}~|\mathscr{F}_i|\geq 2$, and
   $\emptyset$ {otherwise}. Then $|L_i|$ is the number of vertices which appear
   simultaneously in at least two edges in $\mathscr{F}_i$ when
   $|\mathscr{F}_i|\geq 2$. If each vertex contained in the edges of
   $\mathscr{F}_i$ appears simultaneously in only two edges of
   $\mathscr{F}_i$, then $\mathscr{F}_i$ is called a {\em barrier}.
 Obviously, $|\bigcup_{i=1}^{m-h}L_i|\leq l$. If $|\bigcup_{i=1}^{m-h}L_i|<l$ for
 any $(m-h)$-partition of $\mathscr{F}$, then we say $H\in \mathscr{P}_{h,l}$ and
 call that $H$ satisfies the property $\mathscr{P}_{h,l}$. Clearly, if $H\notin \mathscr{P}_{h,l}$, then there exists a $(m-h)$-partition
$\{\mathscr{F}_1,\mathscr{F}_2,\ldots,\mathscr{F}_{m-h}\}$ of $E(H)$ such that $|\bigcup_{i=1}^{m-h}L_i|=l$, that is, $\bigcup_{i=1}^{m-h}L_i=V$.

Note that for each $1\leq i\leq m-h$,
\begin{eqnarray}
   |L_i|\leq \left\{
\begin{array}{ll}
   \lfloor \frac{r(H)|\mathscr{F}_i|}{2} \rfloor,
   &\mbox{if}~|\mathscr{F}_i|\geq 3              \\
   s(H),                                        &\mbox{if}~|\mathscr{F}_i|=2
   \\
   0,                                                        &\mbox{otherwise}.
   \end{array}\right.
\end{eqnarray} Hence, if $|\mathscr{F}_i|\geq 3$ and $r(H)$ is even, then the
equality in (1) holds if and only if $\mathscr{F}_i$ is a barrier.

\begin{lemma}
\label{2.3}
Let $H=(V,\mathscr{F})$ be a hypergraph with $|V(H)|\geq 7$ and
$\mathscr{F}=\{F_1,F_2,\ldots, F_m\}$ with $m\geq 4$. If $r(H)=3$ and $s(H)=2$,
then $H\in \mathscr{P}_{3,7}$.
\end{lemma}

\begin{proof} Let  $\{\mathscr{F}_1,\mathscr{F}_2,\ldots,\mathscr{F}_{m-3}\}$ be
an $(m-3)$-partition of $\mathscr{F}$ with $|\mathscr{F}_1|\geq
|\mathscr{F}_{2}|\geq \cdots \geq |\mathscr{F}_{m-3}|\ge 1$. Then $2\le
|\mathscr{F}_1|\le 4$ and $\sum_{i=1}^{m-3}(|\mathscr{F}_i|-1)=3$. Since $r(H)=3$
and $s(H)=2$, we have $|L_i|\leq 2(|\mathscr{F}_i|-1)$ for $1\leq i\leq m-3$ by
$(1)$. So $|\bigcup_{i=1}^{m-h}L_i|\le \sum_{i=1}^{m-3}|L_i|\leq 2\sum_{i=1}^{m-3}(|\mathscr{F}_i|-1)=6<7$, and hence $H\in \mathscr{P}_{3,7}$.
\end{proof}

Let $G=(V,E)$ be a (hyper)graph. Let $\phi$ be an edge-coloring of $G$ with the
color set $K=\{c_1,c_2,\ldots,c_t\}$. Note that coloring here need not be proper.
Suppose $\phi^{-1}(c_i)\neq \emptyset$ for each $c_i\in K$. Let $E(c_i)=\{e\in
E(G)~|~\phi(e)=c_i\}$ and $D(c_i)=|E(c_i)|-1$ for $c_i\in K$. Set
$D(G)=\sum_{c_i\in K} D(c_i)$. Then $D(c_i)$ is the repeated number of the color
$c_i$ and $D(G)$ is the
repeated number of all colors in $G$. Obviously, $|K|=|E(G)|-D(G)$. In the
following, we will give a necessary and sufficient condition for the bounds of
$Ar_G(\Psi)$.

\begin{lemma}
\label{2.4}
Let $G=(V,E)$ be a simple graph with $|E(G)|=m$. Let $\Psi=\{G_1,\ldots,G_l\}$
such that $E(G)=\cup^l_{i=1}E(G_i)$, where $G_i$ is a subgraph of $G$ for $1\leq
i\leq l$. Suppose $h$ is an integer with $0\leq h\leq m-1$. Then $H_{G,\Psi}\in
\mathscr{P}_{h,l}$ if and only if $Ar_G(\Psi)\leq |E(G)|-(h+1)$.
\end{lemma}

\begin{proof} Suppose that $E(G)=\{e_1,\ldots,e_m\}$,
$V(H_{G,\Psi})=\{x_1,\ldots,x_l\}$ and $E(H_{G,\Psi})=\{F_1,F_2,\ldots,F_m\}$,
where each $x_i$ corresponds to $G_i$ in $\Psi$ for $1\leq i\leq l$ and each
$F_j$ corresponds to $e_j$ for $1\leq j\leq m$.

Let $\phi$ be an edge-coloring of $G$ with the color set
$K=\{c_1,c_2,\ldots,c_{m-h}\}$ such that $\phi^{-1}(c_i)\not=\emptyset$. Consider
the edge-coloring $\phi^*$ of $H_{G,\Psi}$: $\phi^*(F_j)=\phi(e_j)$ for $1\leq
j\leq m$. Then $D_{\phi}(G)=D_{\phi^*}(H_{G,\Psi})=h$. Denote
$\mathscr{F}_i=\{F\in E(H_{G,\Psi})~|~\phi^*(F)=c_i\}$, $1\le i\le m-h$. Then
$\{\mathscr{F}_1,\mathscr{F}_2,\ldots,\mathscr{F}_{m-h}\}$ is a $(m-h)$-partition
of $E(H_{G,\Psi})$.

If $H_{G,\Psi}\in \mathscr{P}_{h,l}$, then $|\bigcup_{i=1}^{m-h}L_i|<l$, and thus $\bigcup_{i=1}^{m-h}L_i\subset V(H_{G,\Psi})$. Let $x_p\in V(H_{G,\Psi})\setminus (\bigcup_{i=1}^{m-h}L_i)$. Then each edge of $F_{H_{G,\Psi}}(x_p)$ is colored
with distinct colors in $H_{G,\Psi}$, and then every edge of $E(G_p)$ is colored
with distinct colors in $G$. It follows that there exists a rainbow $G_p$ in any
edge-coloring of $G$ by using $m-h=|E(G)|-D_{\phi}(G)$ colors. So $Ar_G(\Psi)\leq
(|E(G)|-D_{\phi}(G))-1=|E(G)|-(h+1)$.

Conversely, we suppose that $H_{G,\Psi}\notin \mathscr{P}_{h,l}$. Then there
exists a $(m-h)$-partition
$\{\mathscr{F}'_1,\mathscr{F}'_2,\ldots,\mathscr{F}'_{m-h}\}$ of $E(H_{G,\Psi})$
such that $\bigcup_{i=1}^{m-h}L'_i=V(H_{G,\Psi})$. Let $\phi'$ be
an edge-coloring of $H_{G,\Psi}$ such that $\phi'^{-1}(c_i)=\mathscr{F}'_i$ for
$1\leq i\leq m-h$. Then $D_{\phi'}(H_{G,\Psi})=h$.

Now we give an edge-coloring $\phi''$ of $G$ as follows:
$\phi''(e_j)=\phi'(F_j)$, where
each $F_j$ corresponds to $e_j$ $(1\leq j\leq m)$. Then
$D_{\phi''}(G)=D_{\phi'}(H_{G,\Psi})=h$. Note that, for any $x_p\in
V(H_{G,\Psi})$, $x_p\in L'_i$ for some $i$ ($1\leq i\leq m-h$), and hence there are at least two edges of $F_{H_{G,\Psi}}(x_p)$ colored with
the same color in $H_{G,\Psi}$, which implies at least two edges in each $G_p$ are colored with the same color. So, $G$ contains
no rainbow-$\Psi$ graph of $G$.
Therefore, $Ar_G(\Psi)\geq |E(G)|-D_{\phi''}(G)=|E(G)|-h$.
\end{proof}

In order to obtain the bounds of $Ar_G(\Psi)$, we need to check that whether a
hypergraph $H_{G,\Psi}$ satisfies the property $\mathscr{P}_{h,l}$ or not by
Lemma \ref{2.4}. So we give the following algorithm.

\begin{algorithm}
\label{2.5}
Determine a hypergraph $H_{G,\Psi}$ with the property $\mathscr{P}_{h,l}$ or not.

\indent {\bf Input:} ~~A hypergraph $H_{G,\Psi}=(V,\mathscr{F})$ with $|V|=l$ and
$\mathscr{F}=\{F_1,F_2,\ldots,F_m\}$,

an integer $h$ with $0\leq h\leq m-1$.

\indent {\bf Output:} ~~$H_{G,\Psi}\in \mathscr{P}_{h,l}$ or $H_{G,\Psi}\notin
\mathscr{P}_{h,l}$.

1: set $Q:=\emptyset$

2: {\bf for} each decreasing positive integer sequence $(f_1,f_2,\ldots,f_{m-h})$
such that

$~~~~~~~~~~\sum_{i=1}^{m-h}f_i=m$ {\bf do}

3: ~~~~~append $(f_1,f_2,\ldots,f_{m-h})$ to $Q$

4: {\bf end for}

5: {\bf while} $Q$ is nonempty {\bf do}

6: ~~~~~consider a sequence $(f_1,f_2,\ldots,f_{m-h})$ of $Q$

7: ~~~~~{\bf for} each $(m-h)$-partition
$\{\mathscr{F}_1,\mathscr{F}_2,\ldots,\mathscr{F}_{m-h}\}$ of $\mathscr{F}$ such
that $|\mathscr{F}_t|=f_t$,

$~~~~~~~~~~~~~~1\leq t\leq m-h$ {\bf do}

8: ~~~~~~~~~set $i:=1$

9: ~~~~~~~~~{\bf for} $i<m-h+1$ {\bf do}

10:~~~~~~~~~~~~~~{\bf if} $f_i\geq 2$ {\bf then}

11:~~~~~~~~~~~~~~~~~~$L_i:=\{v\in F_t\cap F_j~|~F_t,F_j\in
\mathscr{F}_i~\mbox{with}~t\neq j\}$

12:~~~~~~~~~~~~~~{\bf else}

13:~~~~~~~~~~~~~~~~~~$L_i:=\emptyset$

14:~~~~~~~~~~~~~~{\bf end if}

15:~~~~~~~~~{\bf end for}

16:~~~~~~~~~{\bf if} $|\bigcup_{i=1}^{m-h}L_i|=l$ {\bf then}

17:~~~~~~~~~~~~return $H_{G,\Psi}\notin \mathscr{P}_{h,l}$

18:~~~~~~~~~{\bf break}

19:~~~~~~~~~{\bf end if}

20:~~~~~{\bf end for}

21:~~~~~remove $(f_1,f_2,\ldots,f_{m-h})$ from $Q$

22: {\bf end while}

23: return $H_{G,\Psi}\in \mathscr{P}_{h,l}$

\end{algorithm}


\section{Anti-Ramsey number $Ar(P_{n,k}, C_5)$}


In this section, we will determine the anti-Ramsey number $Ar(P_{n,k}, C_5)$ for
$1\leq i\leq \lfloor \frac{n-1}{2} \rfloor$. If $G$ contains no $C_5$, then
$Ar(P_{n,k}, C_5)=|E(P_{n,k})|=3n$. Thus by Proposition \ref{1.1}(i), we can
assume that $n\in\{3,5\}$ or $k\in \{2,\frac{n}{5},\frac{2n}{5},\frac{n-1}{2}\}$.


\begin{lemma}
\label{3.1}
$Ar(P_{3,1}, C_5)=7$.
\end{lemma}


\begin{proof} Let $G^{1}_i:=u_iu_{i+1}v_{i+1}v_{i+2}u_{i+2}u_i,
G^{2}_i:=u_iu_{i+1}v_{i+1}v_{i+2}v_{i}u_i$ for $0\leq i\leq 2$. Then $\Psi=\{
G^{1}_i, G^{2}_i~|~0\leq i\leq 2\}$ is the set of all $C_5$ in $P_{3,1}$, and
then $Ar(P_{3,1}, C_5)=Ar_{P_{3,1}}(\Psi)$.

First we show that $Ar(P_{3,1}, C_5)\geq 7$. Let $\phi: E(P_{3,1})\rightarrow
\{1,2,\ldots,7\}$ be an edge-coloring of $P_{3,1}$ such that $\phi(u_iv_i)=1$ for
$0\leq i\leq 2$ and every other edges obtains distinct colorings from
$\{2,\ldots,7\}$ (see Fig.~1$(a)$). Note that each $C_5\in \Psi$ contains two
spokes $u_iv_i$ and $u_{i+1}v_{i+1}$ for some $0\le i\le 2$, and hence $P_{3,1}$
has no rainbow $C_5$ under the coloring $\phi$. So, $Ar(P_{3,1}, C_5)\geq 7$.

Next we show that $Ar(P_{3,1}, C_5)\leq 7$. Let
$V(H_{P_{3,1},\Psi})=\{x_i,y_i~|~0\leq i\leq 2\}$ and
$E(H_{P_{3,1},\Psi})=\{F_0,\ldots,F_{8}\}$, where $x_i$ and $y_i$ correspond to $G^{1}_i$ and $G^{2}_i$ in $\Psi$, respectively, each $F_i$ corresponds
to $u_iu_{i+1}$, $F_{i+3}$ corresponds to $u_iv_i$, and
$F_{i+6}$ corresponds to $v_iv_{i+1}$ for $0\leq i\leq 2$.
Note that $u_iu_{i+1}$ is contained in $G^{1}_i$, $G^{1}_{i+1}$ and
$G^{2}_{i}$, $u_iv_i$ is contained in $G^{1}_{i+1}$, $G^{1}_{i+2}$,
$G^{2}_{i}$ and $G^{2}_{i+2}$, $v_iv_{i+1}$ is contained in
$G^{1}_{i+2}$, $G^{2}_{i+1}$ and $G^{2}_{i+2}$ for $0\leq i\leq 2$. Thus
$F_i=\{x_i,x_{i+1},y_i\}$, $F_{i+3}=\{x_{i+1},x_{i+2},y_{i},y_{i+2}\}$,
$F_{i+6}=\{x_{i+2},y_{i+1},y_{i+2}\}$ for $0\leq i\leq 2$, and so
$r(H_{P_{3,1},\Psi})=4$ and $s(H_{P_{3,1},\Psi})=2$. By Lemma \ref{2.2}, we have
$Ar(P_{3,1}, C_5)=Ar_{P_{3,1}}(\Psi)\leq |E(P_{3,1})|-\lceil \frac{2|\Psi|}{4+2}
\rceil=9-2=7$.
\end{proof}


\begin{figure}[!htb]
\centering
{\includegraphics[height=0.35\textwidth]{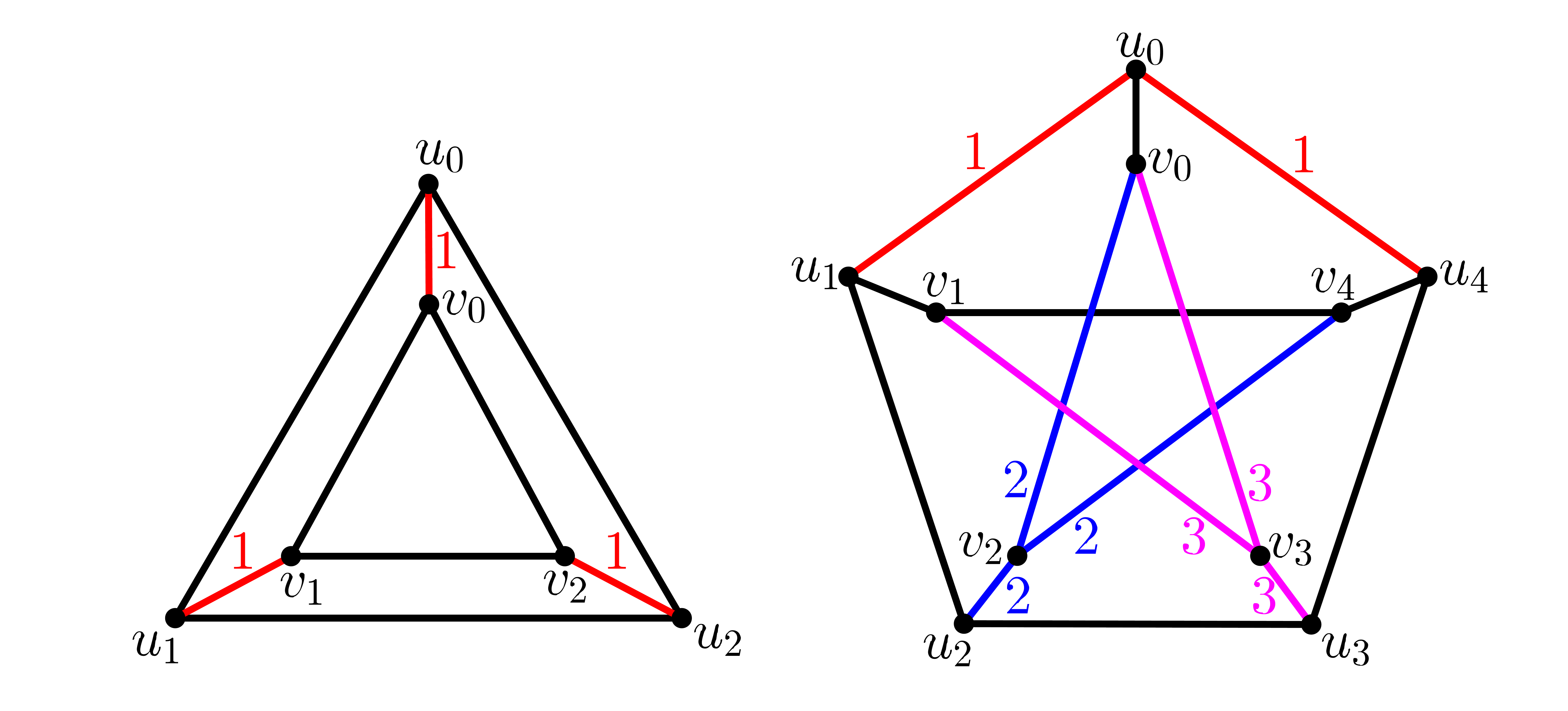}}

\vskip -.3cm

~$(a)$~$P_{3,1}$~~~~~~~~~~~~~~~~~~~~~~~~~~~~~$(b)$~$P_{5,2}$

\vskip .5cm

Fig.~1~$P_{3,1}$ and $P_{5,2}$ have no rainbow $C_5$ under the coloring $\phi$

\end{figure}


\begin{lemma}
\label{3.2}
$Ar(P_{5,1}, C_5)=13$.
\end{lemma}


\begin{proof} Let $P'_{5,1}=P_{5,1}-E'$, where $E'=\{u_iv_i~|~0\leq i\leq 4\}$, and
let $G_1:=u_0u_1u_2u_3u_4u_0, G_2:=v_0v_1v_2v_3v_4v_0$. Then $\Psi'=\{G_1, G_2\}$
is the set of all copies of $C_5$ in $P'_{5,1}$, and then $Ar(P'_{5,1},
C_5)=Ar_{P'_{5,1}}(\Psi')$. Note that each edge in $E(G_i)$ is only contained in
$G_{i}$. Thus $H_{P'_{5,1},\Psi'}$ is a graph obtained from two isolated vertices
by adding five loops on each isolated vertex, and so $M(H_{P'_{5,1},\Psi'})=0$.
By Lemma \ref{2.1}, we have $Ar(P'_{5,1},
C_5)=Ar_{P'_{5,1}}(\Psi')=|E(P'_{5,1})|-|\Psi'|+M(H_{P'_{5,1},\Psi'})=8$. Since
there is no copy of $C_5$ in $P_{5,1}$ containing any edge of $E'$,
$Ar(P_{5,1},C_5)=|E'|+Ar(P'_{5,1},C_5)=13$.
\end{proof}


\begin{lemma}
\label{3.3}
$Ar(P_{5,2}, C_5)=10$.
\end{lemma}

\begin{proof} Let $G^{1}_i:=u_iu_{i+1}v_{i+1}v_{i+3}v_iu_i,
G^{2}_i:=u_iu_{i+1}u_{i+2}v_{i+2}v_{i}u_i$ for $0\leq i\leq 4$, and let
$G^{3}:=u_0u_1u_2u_3u_4,$ $ G^{4}:=v_0v_2v_4v_1v_3v_0$. Then $\Psi=\{ G^{1}_i,
G^{2}_i, G^{3},G^{4}~|~0\leq i\leq 4\}$ is the set of all copies of $C_5$ in
$P_{5,2}$. Hence $Ar(P_{5,2}, C_5)=Ar_{P_{5,2}}(\Psi)$.

First we show that $Ar(P_{5,2}, C_5)\geq 10$. Let $\phi: E(P_{5,2})\rightarrow
\{1,2,\ldots,10\}$ be an edge-coloring of $P_{5,2}$ defined as follows:
\begin{eqnarray}
   \phi(e)=\left\{
\begin{array}{ll}
   1,                          &\mbox{if}~e\in E(u_0)\setminus \{u_0v_0\};
   \nonumber   \\
   2,                          &\mbox{if}~e\in E(v_2);
   \nonumber   \\
   3,                          &\mbox{if}~e\in E(v_3),
   \nonumber
   \end{array}\right.
\end{eqnarray}
and every other edges obtains distinct colorings from $\{4,\ldots,10\}$ (see
Fig.~1$(b)$). Note that each $C_5\in \Psi$ contains the edges
$\{u_0u_1,u_0u_4\}$, or two edges of  $E(v_i)$ ($i\in \{2,3\}$), and so $P_{5,2}$
has no rainbow $C_5$. Hence, $Ar(P_{5,2}, C_5)\geq 10$.

Next we show that $Ar(P_{5,2}, C_5)\leq 10$. Let
$V(H_{P_{5,2},\Psi})=\{x_i,y_i,z,w~|~0\leq i\leq 4\}$ and
$E(H_{P_{5,2},\Psi})=\{F_0,\ldots,F_{14}\}$, where $x_i$, $y_i$, $z$
and $w$ correspond to $G^{1}_i$, $G^{2}_i$, $G^{3}$ and $G^{4}$ in
$\Psi$, respectively, each $F_i$ corresponds to $u_iu_{i+1}$,
$F_{i+5}$ corresponds to $u_iv_i$, and $F_{i+10}$ corresponds to
$v_iv_{i+2}$ for $0\leq i\leq 4$. Note that each $u_iu_{i+1}$ is
contained in $G^{1}_i$, $G^{2}_i$, $G^{2}_{i+4}$ and $G^{3}$, each spoke $u_iv_i$
is contained in $G^{1}_i$, $G^{1}_{i+4}$, $G^{2}_{i}$ and $G^{2}_{i+3}$, and each
$v_iv_{i+2}$ is contained in $G^{1}_{i+2}$, $G^{1}_{i+4}$, $G^{2}_{i}$ and
$G^{4}$. Thus $F_i=\{x_i,y_i,y_{i+4},z\}$,
$F_{i+5}=\{x_i,x_{i+4},y_{i},y_{i+3}\}$ and $F_{i+10}=\{x_{i+2},x_{i+4},y_i,w\}$
for $0\leq i\leq 4$. It follows that $r(H_{P_{5,2},\Psi})=4$ and
$s(H_{P_{5,2},\Psi})=2$.

Now we will show that $H_{P_{5,2},\Psi}\in \mathscr{P}_{4,12}$. Suppose that $H_{P_{5,2},\Psi}\notin \mathscr{P}_{4,12}$. Then there exists a $11$-partition $\{\mathscr{F}_1,\mathscr{F}_2,\ldots,\mathscr{F}_{11}\}$ of $E(H_{P_{5,2},\Psi})$ with $ |\bigcup_{i=1}^{11}L_i|=12$. Assume that $|\mathscr{F}_1|\geq \cdots \geq |\mathscr{F}_{11}|\ge
1$. Then $2\le |\mathscr{F}_1|\le 5$ and $\sum_{i=1}^{11}(|\mathscr{F}_i|-1)=4$.
If $|\mathscr{F}_2|\leq
2$, then by $(1)$, $|\bigcup_{i=1}^{11}L_i|\le \sum_{i=1}^{11}|L_i|\leq
2|\mathscr{F}_1|+\sum_{i=2}^{11}2(|\mathscr{F}_i|-1)=2\sum_{i=1}^{11}(|\mathscr{F}_i|-1)+2=10<12$,
a contradiction. So $|\mathscr{F}_{2}|\geq 3$, and then $|\mathscr{F}_1|=
|\mathscr{F}_{2}|=3$. By $(1)$, $$12= |\bigcup_{i=1}^{11}L_i|\le \sum_{i=1}^{11}|L_i|\leq
2|\mathscr{F}_1|+2|\mathscr{F}_1|+\sum_{i=3}^{11}2(|\mathscr{F}_i|-1)=2\sum_{i=1}^{11}(|\mathscr{F}_i|-1)+4=12,$$
which implies $|L_i|=2|\mathscr{F}_i|=6$ for $i=1,2$, $L_1\cup
L_2=V(H_{P_{5,2},\Psi})$ and $L_1\cap L_2=\emptyset$. Hence, each $\mathscr{F}_i$
($i\in\{1,2\}$) is a barrier. Without loss of generality, we assume $F_0\in
\mathscr{F}_1$. Then $x_0,y_0,y_4,z\in L_1$, and then $\{F_{10}, F_{11},  F_{13},
F_{14}\}\cap \mathscr{F}_2=\emptyset$ as $y_0\in F_{10}, x_0\in F_{11}\cap
F_{13}, y_4\in F_{14}$. Note that $F(w)=\{F_i~|~{10}\le i\le {14}\}$, and so
$w\notin L_2$ as $|\mathscr{F}_2|=3$. Then $w\in L_1$, and then
$|\mathscr{F}_1|\ge 4$, a contradiction. Therefore, $H_{P_{5,2},\Psi}\in
\mathscr{P}_{4,12}$. By Lemma \ref{2.4}, $Ar(P_{5,2},C_5)=Ar_{P_{5,2}}(\Psi)\leq
|E(P_{5,2})|-(4+1)=10$.
\end{proof}


\begin{lemma}
\label{3.4}
$Ar(P_{10,2}, C_5)=22$.
\end{lemma}


\begin{proof} Let $G^1_i:=u_iu_{i+1}u_{i+2}v_{i+2}v_iu_i$ for $0\leq i\leq 9$,
and let $G^2:=v_0v_2v_4v_6v_8v_0, G^3:=v_1v_3v_5v_7v_9v_1$. Then $\Psi=\{G^1_i,
G^2, G^3~|~0\leq i\leq 9\}$ is the set of all copies of $C_5$ in $P_{10,2}$, and
then $Ar(P_{10,2}, C_5)=Ar_{P_{10,2}}(\Psi)$.

Note that each edge $u_iu_{i+1}$ is contained in $G^1_i$ and $G^1_{i+9}$; each
spoke $u_iv_i$ is contained in $G^1_{i}$ and $G^1_{i+8}$; each $v_iv_{i+2}$ is
contained in $G^1_{i}$ and $G^2$ for $i\in \{0,2,4,6,8\}$; each edge $v_iv_{i+2}$
is contained in $G^1_{i}$ and $G^3$ for $i\in \{1,3,5,7,9\}$. Thus
$H_{P_{10,2},\Psi}$ is a simple graph of order $12$ (see Fig.~2). Since there
exist four vertex-disjoint cycles $w_0w_1w_2w_0$, $w_3w_4w_5w_3$,
$w_6w_8w_{11}w_6$ and $w_7w_9w_{10}w_7$ in $H_{P_{10,2},\Psi}$,
$M(H_{P_{10,2},\Psi})=4$. By Lemma \ref{2.1}, $Ar(P_{10,2},
C_5)=Ar_{P_{10,2}}(\Psi)=|E(P_{10,2})|-|\Psi|+M(H_{P_{10,2},\Psi})=22$.
\end{proof}

\begin{figure}[!htb]
\centering
{\includegraphics[height=0.35\textwidth]{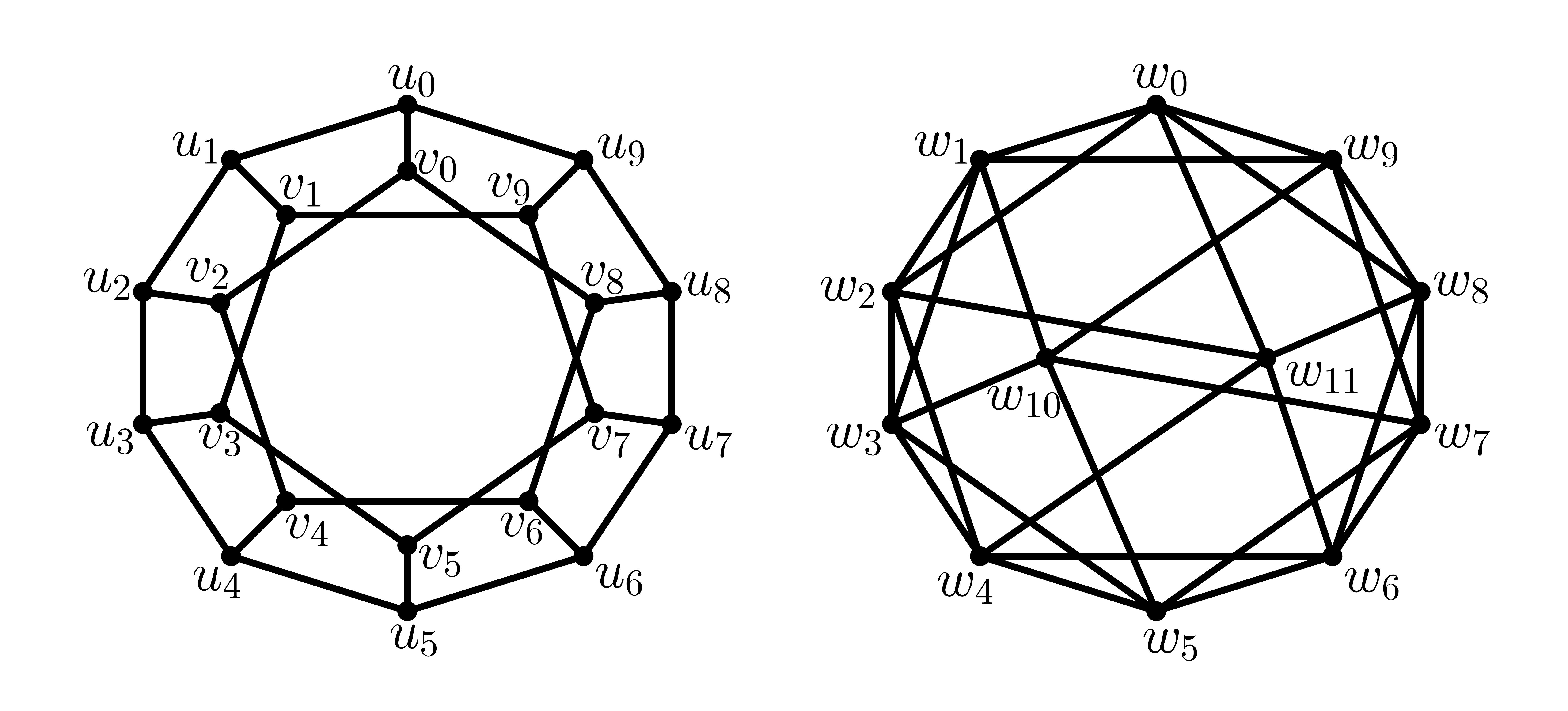}}

Fig.~2 ~ $P_{10,2}$ and $H_{P_{10,2},\Psi}$
\end{figure}


\begin{lemma}
\label{3.5}
$Ar(P_{n,2}, C_5)=\lfloor \frac{7n}{3} \rfloor$ for $n\geq 6$ and $n\neq 10$.
\end{lemma}


\begin{proof} Let $G_i:=u_iu_{i+1}u_{i+2}v_{i+2}v_iu_i$ for $0\leq i\leq n-1$.
Then $\Psi=\{G_i~|~0\leq i\leq n-1\}$ is the set of all copies of $C_5$ in
$P_{n,2}$, and then $Ar(P_{n,2}, C_5)=Ar_{P_{n,2}}(\Psi)$. Note that each edge
$u_iu_{i+1}$ is contained in $G_i$ and $G_{i+n-1}$; $u_iv_i$ is
contained in $G_i$ and $G_{i+n-2}$; $v_iv_{i+2}$ is contained in $G_i$
for $0\leq i\leq n-1$. Thus $H_{P_{n,2},\Psi}$ is a graph obtained from
$C_n=w_0w_1\cdots w_{n-1}w_0$ by adding a loop on $w_i$ and edge
$w_iw_{i+2}$ for all $0\leq i\leq n-1$. Since there exist $\lfloor \frac{n}{3}
\rfloor$ vertex-disjoint cycles $w_{3i}w_{3i+1}w_{3i+2}w_{3i}$ in
$H_{P_{n,2},\Psi}$ for $0\leq i\leq \lfloor \frac{n}{3} \rfloor-1$, we have
$M(H_{P_{n,2},\Psi})=\lfloor \frac{n}{3} \rfloor$. By Lemma \ref{2.1},
$Ar(P_{n,2},
C_5)=Ar_{P_{n,2}}(\Psi)=|E(P_{n,2})|-|\Psi|+M(H_{P_{n,2},\Psi})=\lfloor
\frac{7n}{3} \rfloor$.
\end{proof}


\begin{lemma}
\label{3.6}
$Ar(P_{n,k}, C_5)=\frac{14n}{5}$ for $k\geq 3$ and
$k\in\{\frac{n}{5},\frac{2n}{5}\}$.
\end{lemma}


\begin{proof} Denote $P_{n,k}':=P_{n,k}-E'$, where $E'=\{u_iu_{i+1}, u_iv_i~|~0\leq
i\leq n-1\}$.
Let $G_i:=v_iv_{i+k}v_{i+2k}v_{i+3k}v_{i+4k}v_i$ for $0\leq i\leq \frac{n}{5}-1$.
Then $ \Psi'=
   \{G_i~|~0\leq i\leq \frac{n}{5}-1\}$ is the set of all copies of $C_5$ in
   $P'_{n,k}$ for $k\in\{\frac{n}{5},\frac{2n}{5}\}$, and then $Ar(P'_{n,k},
   C_5)=Ar_{P'_{n,k}}(\Psi')$. Note that each edge in $P_{n,k}'$ is only
   contained in some $G_{i}$. Thus $H_{P'_{n,k},\Psi'}$ is a graph obtained from
   $k$ isolated vertices by adding five loops on each isolated vertex. Then
   $M(H_{P'_{n,k},\Psi'})=0$. By Lemma \ref{2.1}, $Ar(P'_{n,k},
   C_5)=Ar_{P'_{n,k}}(\Psi')=|E(P'_{n,k})|-|\Psi'|+M(H_{P'_{n,k},\Psi'})=\frac{4n}{5}$
   for $k\in \{\frac{n}{5},\frac{2n}{5}\}$. Since there is no copy of $C_5$ in
   $P_{n,k}$ containing any edge of $E'$, we have
   $Ar(P_{n,k},C_5)=|E'|+Ar(P'_{n,k},C_5)=\frac{14n}{5}$.
\end{proof}


\begin{lemma}
\label{3.7}
$Ar(P_{n,k}, C_5)=\lfloor \frac{7n}{3} \rfloor$ for $k=\frac{n-1}{2}\ge 3$.
\end{lemma}


\begin{proof} Note that $k=\frac{n-1}{2}\ge 3$. Hence $n\ge 7$, moreover, $n\ge
11$ if $n\equiv 2\pmod 3$.

Let $G_i:=u_iu_{i+1}v_{i+1}v_{i+\frac{n+1}{2}}v_iu_i$ for $0\leq i\leq n-1$. Then
$\Psi=\{G_i~|~0\leq i\leq n-1\}$ is the set of all copies of $C_5$ in $P_{n,k}$,
and then $Ar(P_{n,k}, C_5)=Ar_{P_{n,k}}(\Psi)$.
Note that each $u_iu_{i+1}$ is contained in $G_i$, $u_iv_i$ is
contained in $G_i$ and $G_{i+n-1}$, $v_iv_{i+\frac{n-1}{2}}$ is
contained in $G_{i+\frac{n-1}{2}}$ and $G_{i+n-1}$. Thus $H_{P_{n,k},\Psi}$ is a
graph obtained from the cycle $w_0w_1\cdots w_{n-1}w_0$ by adding a loop on $w_i$ and an edge $w_{i+\frac{n-1}{2}}w_{i+n-1}$ for all $0\leq i\leq n-1$.
Then $C_j=w_jw_{j+\frac{n-1}{2}}w_{j+\frac{n+1}{2}}w_j$ is a $3$-cycle in
$H_{P_{n,k},\Psi}$ for $0\leq j\leq n-1$. Let $\mathscr{C}$ be the set of
vertex-disjoint 3-cycles in $H_{P_{n,k},\Psi}$. Then \begin{eqnarray}
  \mathscr{C}=\left\{
\begin{array}{ll}
   \{C_{3i}~|~0\leq i\leq \frac{n}{3}-1\},
   &\mbox{if}~n\equiv 0\pmod 3;                     \\          \nonumber
   \{C_{2+3i},C_{n-3i}~|~0\leq i\leq \frac{n-7}{6}\},
   &\mbox{if}~n\equiv 1\pmod 3;                     \\          \nonumber
   \{C_{0},C_{2+3i},C_{n-2-3i}~|~0\leq i\leq \frac{n-11}{6}\},  &otherwise.
   \nonumber
   \end{array}\right.
\end{eqnarray}
This implies that $M(H_{P_{n,k},\Psi})=\lfloor \frac{n}{3} \rfloor$. By Lemma
\ref{2.1}, we have
$Ar(P_{n,k},
C_5)=Ar_{P_{n,k}}(\Psi)=|E(P_{n,k})|-|\Psi|+M(H_{P_{n,k},\Psi})=\lfloor
\frac{7n}{3} \rfloor$.
\end{proof}

By Lemmas \ref{3.1}-\ref{3.7}, we complete the proof of Theorem \ref{1.2}.

\section{Anti-Ramsey number $Ar(P_{n,k}, C_6)$}

In this section, we will determine the anti-Ramsey number $Ar(P_{n,k}, C_6)$. If
$G$ contains no $C_6$, then $Ar(P_{n,k}, C_6)=|E(P_{n,k})|=3n$. Thus by
Proposition \ref{1.1}$(ii)$, it is only necessary to consider the anti-Ramsey
number $Ar(P_{n,k}, C_6)$ when $k\in
\{1,3,\frac{n}{6},\frac{n-1}{3},\frac{n+1}{3},\frac{n-2}{2}\}$. Moreover, if
$k=2$, then $n\in\{5,6,7,12\}$.

\subsection{$k=1$}

In this subsection, we focus on the case that $k=1$. Then $n\ge 3$.

\begin{lemma}
\label{3.8}
$Ar(P_{3,1}, C_6)=7$.
\end{lemma}

\begin{proof} Let $G_i:=u_iu_{i+1}u_{i+2}v_{i+2}v_{i+1}v_iu_i$ for $0\leq i\leq
2$. Then $\Psi=\{G_0,G_1,G_2\}$ is the set of all copies of $C_6$ in $P_{3,1}$,
and then $Ar(P_{3,1}, C_6)=Ar_{P_{3,1}}(\Psi)$. Note that each $u_iu_{i+1}$ is
contained in $G_i$ and $G_{i+2}$, $u_iv_i$ is contained in $G_{i}$ and
$G_{i+1}$, and $v_iv_{i+1}$ is contained in $G_i$ and $G_{i+2}$ for $0\leq
i\leq 2$. Thus $H_{P_{3,1},\Psi}$ is a graph obtained from the 3-cycle
$w_0w_1w_2w_0$ by adding two parallel edges on each edge of $C_3$. Then
$M(H_{P_{3,1},\Psi})=1$. By Lemma \ref{2.1},
$Ar(P_{3,1},C_6)=Ar_{P_{3,1}}(\Psi)=|E(P_{3,1})|-|\Psi|+M(H_{P_{3,1},\Psi})=7$.
\end{proof}


\begin{figure}[!htb]
\centering
{\includegraphics[height=0.35\textwidth]{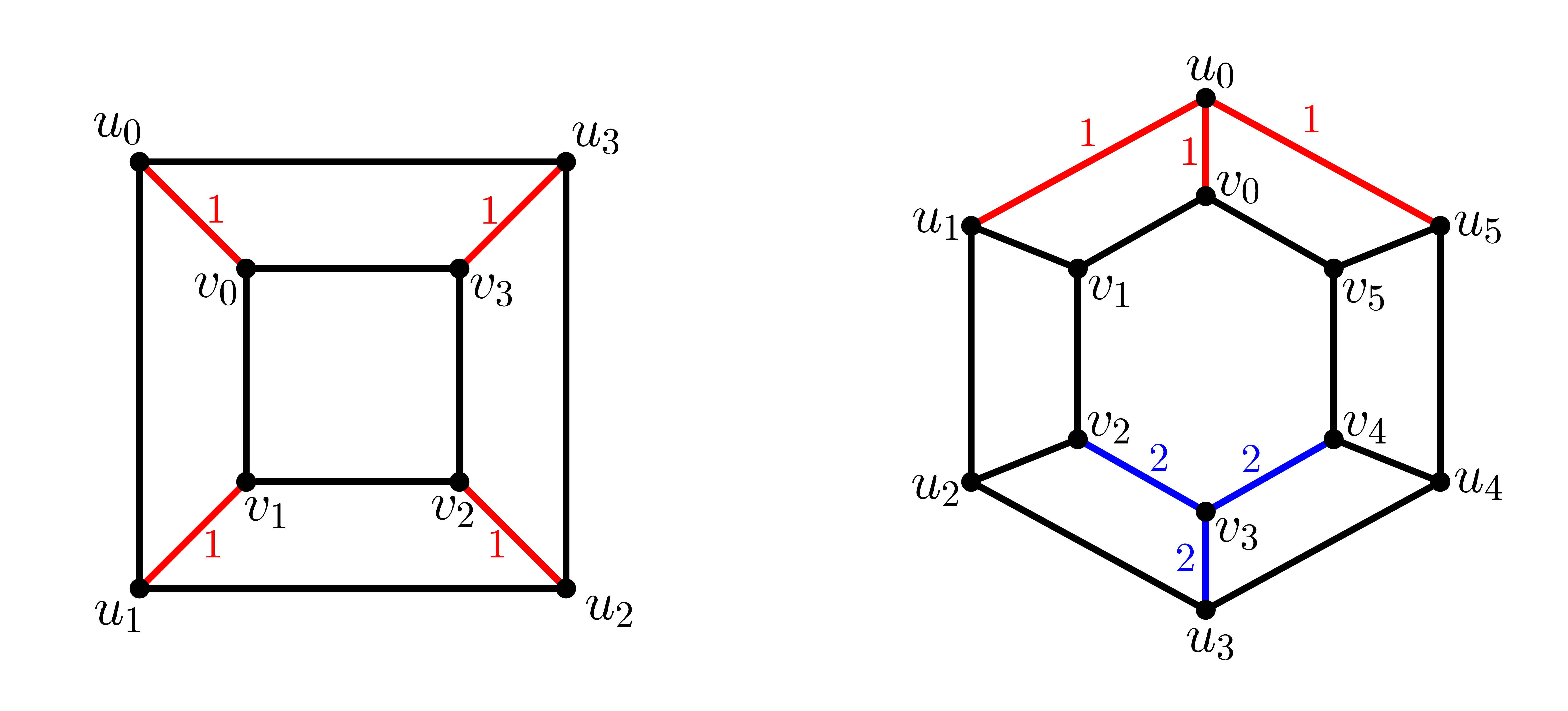}}

\vskip -.1cm

$(a)$~$P_{4,1}$~~~~~~~~~~~~~~~~~~~~~~~~~~~~~~~~~~~~~~~~~~~$(b)$~$P_{6,1}$

\vskip .2cm

Fig.~3 ~$P_{4,1}$ and $P_{6,1}$ have no rainbow $C_6$ under the coloring $\phi$.

\end{figure}

\begin{lemma}
\label{3.9}
$Ar(P_{4,1}, C_6)=9$.
\end{lemma}

\begin{proof} Let $G^1_i:=u_iu_{i+1}v_{i+1}v_{i+2}v_{i+3}v_iu_i$,
$G^2_i:=u_iu_{i+1}u_{i+2}v_{i+2}v_{i+1}v_iu_i$, $G^3_i:=u_iu_{i+1}$ $u_{i+2}v_{i+2}v_{i+3}v_iu_i$ and
$G^4_i:=u_iu_{i+1}u_{i+2}u_{i+3}v_{i+3}v_iu_i$ for $0\leq i\leq 3$. Then
$\Psi=\{G^j_i~|~0\leq i\leq 3, 1\leq j\leq 4\}$ is the set of all $C_6$ in
$P_{4,1}$, and then $Ar(P_{4,1}, C_6)=Ar_{P_{4,1}}(\Psi)$.

Let $\phi: E(P_{4,1})\rightarrow \{1,2,\ldots,9\}$ be an edge-coloring of
$P_{4,1}$ such that $\phi(u_iv_i)=1$ for $0\leq i\leq 3$ and every other edges
having distinct colorings from $\{2,\ldots,9\}$ (see Fig.~3$(a)$). Note that each
$C_6\in \Psi$ contains two spokes, and hence $P_{4,1}$ has no rainbow $C_6$ under
the coloring $\phi$. So, $Ar(P_{4,1}, C_6)\geq 9$.

Now we will show that $Ar(P_{4,1}, C_6)\geq 9$. Let
$V(H_{P_{4,1},\Psi})=\{x_i,y_i,z_i,w_i~|~0\leq i\leq 3\}$ and
$E(H_{P_{4,1},\Psi})=\{F_0,\ldots,F_{11}\}$, where $x_i$, $y_i$,
$z_i$, $w_i$ correspond to $G^1_i$, $G^2_i$, $G^3_i$, $G^4_i$ in
$\Psi$, respectively, each $F_i$ corresponds to $u_iu_{i+1}$,
$F_{i+4}$ corresponds to $u_iv_i$ and $F_{i+8}$ corresponds to $v_iv_{i+1}$ for $0\leq i\leq 3$.
Note that each $u_iu_{i+1}$ is contained in $G^1_i$, $G^2_{i}$, $G^2_{i+3}$,
$G^3_{i}$, $G^3_{i+3}$, $G^4_i$, $G^4_{i+2}$ and $G^4_{i+3}$, $u_iv_i$
is contained in $G^1_i$, $G^1_{i+3}$, $G^2_{i}$, $G^2_{i+2}$, $G^3_{i}$,
$G^3_{i+2}$, $G^4_i$ and $G^4_{i+1}$, and $v_iv_{i+1}$ is contained in
$G^1_{i+1}$, $G^1_{i+2}$, $G^1_{i+3}$, $G^2_{i}$, $G^2_{i+3}$, $G^3_{i+1}$,
$G^3_{i+2}$ and $G^4_{i+1}$ for $0\leq i\leq 3$. It follows that
$F_i=\{x_i,y_i,y_{i+3},z_i,z_{i+3},w_i,w_{i+2},w_{i+3}\}$,
$F_{i+4}=\{x_i,x_{i+3},y_i,y_{i+2},z_i,z_{i+2},w_i,w_{i+1}\}$ and
$F_{i+8}=\{x_{i+1},x_{i+2},x_{i+3},y_i,y_{i+3},z_{i+1},z_{i+2},w_{i+1}\}$ for
$0\leq i\leq 3$, and thus $r(H_{P_{4,1},\Psi})=8$. Obviously,
$s(H_{P_{4,1},\Psi})=4$. By Lemma \ref{2.2}, we have $Ar(P_{4,1},
C_6)=Ar_{P_{4,1}}(\Psi)\leq |E(P_{4,1})|-\lceil \frac{2|\Psi|}{8+4} \rceil=9$.
\end{proof}


\begin{lemma}
\label{3.10}
$Ar(P_{6,1}, C_6)=14$.
\end{lemma}


\begin{proof} Let $G^1_i:=u_iu_{i+1}u_{i+2}v_{i+2}v_{i+1}v_iu_i$ for $0\leq i\leq
5$, $G^2:=u_0u_1u_2u_3u_4u_5u_6u_0$ and $G^3:=v_0v_1v_2v_3v_4v_5v_6v_0$. Then
$\Psi=\{G^1_i,G^2,G^3~|~0\leq i\leq 5\}$ is the set of all copies of $C_6$ in
$P_{6,1}$, and $Ar(P_{6,1}, C_6)=Ar_{P_{6,1}}(\Psi)$.

Let $\phi: E(P_{6,1})\rightarrow \{1,2,\ldots,14\}$ be an edge-coloring of
$P_{6,1}$ defined as follows:
\begin{eqnarray}
   \phi(e)=\left\{
\begin{array}{ll}
   1,                          &\mbox{if}~e\in E(u_0);       \nonumber   \\
   2,                          &\mbox{if}~e\in E(v_3),       \nonumber
   \end{array}\right.
\end{eqnarray}
and every other edges obtains distinct colorings from $\{3,\ldots,14\}$ (see
Fig.~3$(b)$). Note that each $C_6\in \Psi$ contains two edges of $E(u_0)$ (or
$E(v_3)$), and so $P_{6,1}$ has no rainbow $C_6$. Hence, $Ar(P_{6,1}, C_6)\geq
14$.

On the other hand, let $V(H_{P_{6,1},\Psi})=\{x_i,y,z~|~0\leq i\leq 5\}$ and
$E(H_{P_{6,1},\Psi})=\{F_0,\ldots,F_{17}\}$, where $x_i$, $y$ and $z$
correspond to $G^1_i$, $G^2$ and $G^3$ in $\Psi$, respectively, each
$F_i$ corresponds to $u_iu_{i+1}$, $F_{i+6}$ corresponds to $u_iv_i$ and each $F_{i+12}$ corresponds to $v_iv_{i+1}$ for
$0\leq i\leq 5$. Note that each $u_iu_{i+1}$ is contained in $G^1_i$,
$G^1_{i+5}$ and $G^2$, $u_iv_i$ is contained in $G^1_i$ and
$G^1_{i+4}$, and $v_iv_{i+1}$ is contained in $G^1_i$, $G^1_{i+5}$ and
$G^3$ for $0\leq i\leq 5$.  It follows that $F_i=\{x_i,x_{i+5},y\}$,
$F_{i+6}=\{x_i,x_{i+4}\}$ and $F_{i+12}=\{x_i,x_{i+5},z\}$ for $0\leq i\leq 5$,
and then $r(H_{P_{6,1},\Psi})=3$ and $s(H_{P_{6,1},\Psi})=2$. By Lemma \ref{2.2},
we have $Ar(P_{6,1}, C_6)=Ar_{P_{6,1}}(\Psi)\leq |E(P_{6,1})|-\lceil
\frac{2|\Psi|}{3+2} \rceil=14$.
\end{proof}


\begin{lemma}
\label{3.11}
$Ar(P_{n,1}, C_6)=\lfloor \frac{5n}{2}\rfloor$ for $n\geq 5$ and $n\neq 6$.
\end{lemma}


\begin{proof} Since $n\geq 5$ and $n\neq 6$, $\Psi=\{G_i~|~0\leq i\leq n-1\}$ is
the set of all copies of $C_6$ in $P_{n,1}$, where
$G_i:=u_iu_{i+1}u_{i+2}v_{i+2}v_{i+1}v_iu_i$. Note that each $u_iu_{i+1}$ is
contained in $G_i$ and $G_{i+n-1}$, $u_iv_i$ is contained in $G_i$ and
$G_{i+n-2}$, $v_iv_{i+1}$ is contained in $G_i$ and $G_{i+n-1}$ for
$0\leq i\leq n-1$. Thus $H_{P_{n,1},\Psi}$ is a graph obtained from
$C_n=w_0w_1\cdots w_{n-1}w_0$ by adding one parallel edge on each edge of $C_n$
and edge $w_iw_{i+2}$ for $0\leq i\leq n-1$. Then $M(H_{P_{n,1},\Psi})=\lfloor
\frac{n}{2}\rfloor$. By Lemma \ref{2.1},
$Ar(P_{n,1},C_6)=|E(P_{n,1})|-|\Psi|+M(H_{P_{n,1},\Psi})=\lfloor
\frac{5n}{2}\rfloor$.
\end{proof}


By Lemmas \ref{3.8}-\ref{3.11}, the proof of Theorem \ref{1.3} is completed.

\subsection{$k=2$}

In this subsection, we focus on the case that $k=2$. Then $n\ge 5$. By
Proposition \ref{1.1}$(ii)$, we have $n\in\{5,6,7,12\}$.

\begin{figure}[!htb]
\centering
{\includegraphics[height=0.35\textwidth]{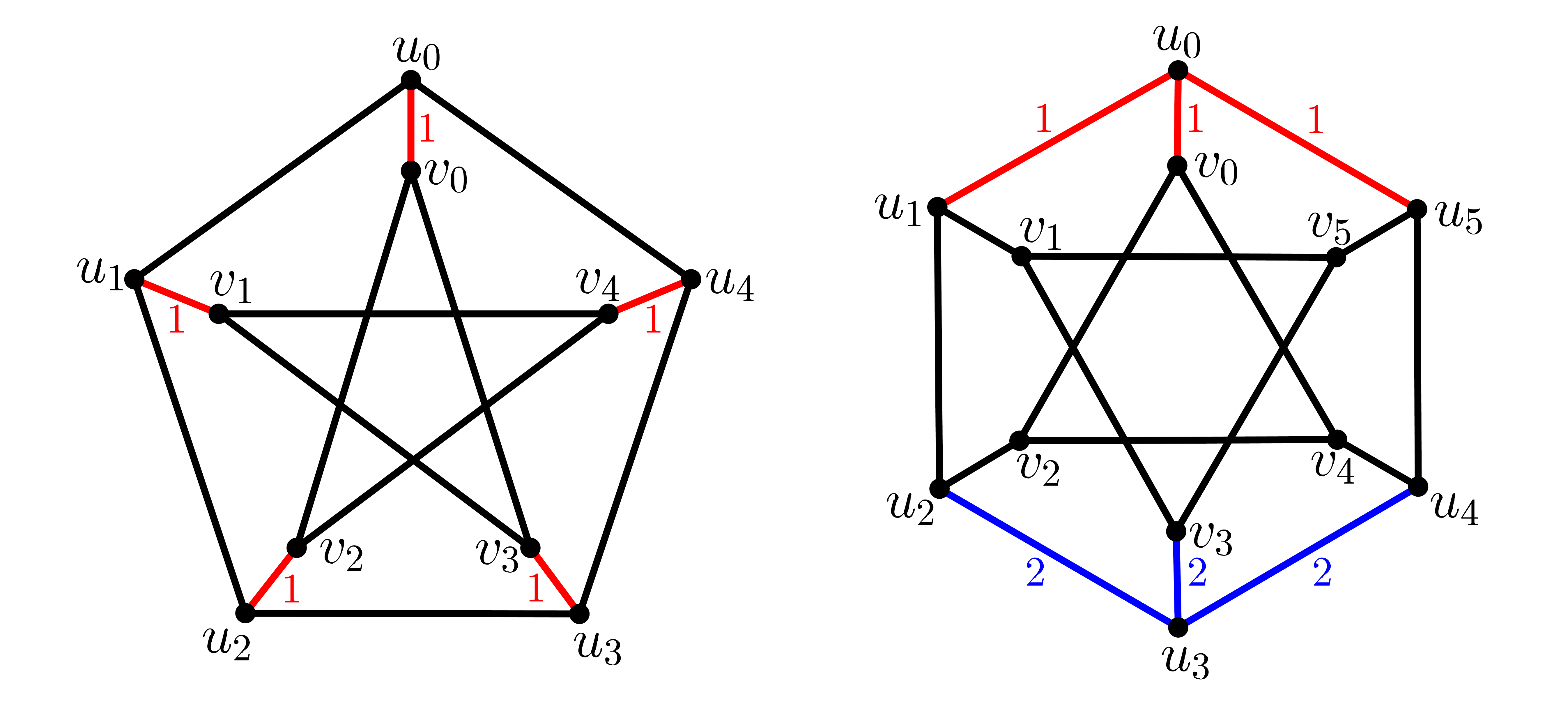}}

\vskip -.1cm

~~~$(a)$~$P_{5,2}$~~~~~~~~~~~~~~~~~~~~~~~~~~~~~~~~~~~~~$(b)$~$P_{6,2}$

\vskip .3cm

Fig.~4 ~$P_{5,2}$ and $P_{6,2}$ have no rainbow $C_6$ under the coloring $\phi$.

\end{figure}


\begin{lemma}
\label{3.12}
$Ar(P_{5,2}, C_6)=11$.
\end{lemma}


\begin{proof} Let $G^1_i=u_iu_{i+1}v_{i+1}v_{i+4}v_{i+2}v_iu_i,
G^2_i=u_iu_{i+1}u_{i+2}u_{i+3}v_{i+3}v_iu_i$ for $0\leq i\leq 4$. Then
$\Psi=\{G^1_i, G^2_i~|~0\leq i\leq 4\}$ is the set of all copies of $C_6$ in
$P_{5,2}$.

Let $\phi: E(P_{5,2})\rightarrow \{1,2,\ldots,11\}$ be an edge-coloring of
$P_{5,2}$ such that $\phi(u_iv_i)=1$ for $0\leq i\leq 4$ and every other edges
obtains distinct colorings from $\{2,\ldots,11\}$ (see Fig.~4$(a)$). Note that
each $C_6\in \Psi$ contains two spokes, and so $P_{5,2}$ has no rainbow $C_6$.
Hence, $Ar(P_{5,2}, C_6)\geq 11$.

Next we will show that $Ar(P_{5,2}, C_6)\leq 11$. Let
$V(H_{P_{5,2},\Psi})=\{x_i,y_i~|~0\leq i\leq 4\}$ and
$E(H_{P_{5,2},\Psi})=\{F_0,\ldots,F_{14}\}$, where $x_i$ and $y_i$
correspond to $G^{1}_i$ and $G^{2}_i$ in $\Psi$, respectively, each
$F_i$ corresponds to $u_iu_{i+1}$, $F_{i+5}$ corresponds to $u_iv_i$, and $F_{i+10}$ corresponds to $v_iv_{i+2}$ for
$0\leq i\leq 4$.
Note that each edge $u_iu_{i+1}$ is contained in $G^1_i$, $G^2_i$, $G^2_{i+3}$
and $G^2_{i+4}$, $u_iv_i$ is contained in $G^1_i$, $G^1_{i+4}$,
$G^2_i$ and $G^2_{i+2}$, $v_iv_{i+2}$ is contained in $G^1_i$,
$G^1_{i+1}$, $G^1_{i+3}$ and $G^2_{i+2}$ for $0\leq i\leq 4$. It follows that $F_i=\{x_i,y_i,y_{i+3},y_{i+4}\}$,
$F_{i+5}=\{x_i,x_{i+4},y_i,y_{i+2}\}$ and
$F_{i+10}=\{x_i,x_{i+1},x_{i+3},y_{i+2}\}$ for $0\leq i\leq 4$, and thus
$r(H_{P_{5,2},\Psi})=4$ and $s(H_{P_{5,2},\Psi})=2$. By Lemma \ref{2.2}, we have
$Ar(P_{5,2},C_6)\leq |E(P_{5,2})|-\lceil \frac{2|\Psi|}{4+2} \rceil=15-\lceil
\frac{10}{3} \rceil =11$.
\end{proof}


\begin{lemma}
\label{3.13}
$Ar(P_{6,2}, C_6)=14$.
\end{lemma}


\begin{proof} Let $G^1_i:=u_iu_{i+1}v_{i+1}v_{i+3}v_{i+5}u_{i+5}u_i$ for $0\leq
i\leq 5$, and $G^2:=u_0u_1u_2u_3u_4u_5u_0$. Then $\Psi=\{ G^1_i,G^2~|~0\leq i\leq
5\}$ is the set of all $C_6$ in $P_{6,2}$, and $Ar(P_{6,2},
C_6)=Ar_{P_{6,2}}(\Psi)$.

Let $\phi: E(P_{6,2})\rightarrow \{1,2,\ldots,14\}$ be an edge-coloring of
$P_{6,2}$ defined as follows:
\begin{eqnarray}
   \phi(e)=\left\{
\begin{array}{ll}
   1,                          &\mbox{if}~e\in E(u_0);       \nonumber   \\
   2,                          &\mbox{if}~e\in E(u_3),       \nonumber
   \end{array}\right.
\end{eqnarray}
and every other edges obtains distinct colorings from $\{3,\ldots,14\}$ (see
Fig.~4$(b)$). Note that each $C_6\in \Psi$ contains two edges of $E(u_i)$ ($i\in
\{0,3\}$), and so $P_{6,2}$ has no rainbow $C_6$. Hence, $Ar(P_{6,2}, C_6)\geq
14$.

Next we will show that $Ar(P_{6,2}, C_6)\leq 14$. Let
$V(H_{P_{6,2},\Psi})=\{x_i,y~|~0\leq i\leq 5\}$ and
$E(H_{P_{6,2},\Psi})=\{F_0,\ldots,F_{17}\}$, where $x_i$ and $y$ correspond to $G^1_i$ and $G^2$ in $\Psi$, respectively, each $F_i$ corresponds to $u_iu_{i+1}$, $F_{i+6}$ corresponds to $u_iv_i$, and
$F_{i+12}$ corresponds to $v_iv_{i+2}$ for $0\leq i\leq 5$. Note that each edge $u_iu_{i+1}$ is contained in $G_i$, $G_{i+1}$ and $G^2$, $u_iv_i$ is contained in $G_{i+1}$ and $G_{i+5}$, $v_iv_{i+2}$ is
contained in $G_{i+3}$ and $G_{i+5}$ for $0\leq i\leq 5$. Hence $F_i=\{x_i,x_{i+1},y\}$, $F_{i+6}=\{x_{i+1},x_{i+5}\}$ and
$F_{i+12}=\{x_{i+3},x_{i+5}\}$ for $0\leq i\leq 5$, and thus
$r(H_{P_{6,2},\Psi})=3$ and $s(H_{P_{6,2},\Psi})=2$. By Lemma \ref{2.3}, $H_{P_{6,2},\Psi}\in \mathscr P_{3,7}$, and then by Lemma
\ref{2.4}, we have $Ar(P_{6,2}, C_6)=Ar_{P_{6,2}}(\Psi)\leq
|E(P_{6,2})|-(3+1)=14$.
\end{proof}


\begin{lemma}
\label{3.14}
$Ar(P_{7,2}, C_6)=17$.
\end{lemma}


\begin{proof} Let $G_i:=u_iu_{i+1}v_{i+1}v_{i+3}v_{i+5}v_iu_i$ for $0\leq i\leq
6$. Then $\Psi=\{ G_i~|~0\leq i\leq 6\}$ is the set of all copies of $C_6$ in
$P_{7,2}$, and then $Ar(P_{7,2}, C_6)=Ar_{P_{7,2}}(\Psi)$.

Let $\phi: E(P_{7,2})\rightarrow \{1,2,\ldots,17\}$ be an edge-coloring of
$P_{7,2}$ defined as follows:
\begin{eqnarray}
   \phi(e)=\left\{
\begin{array}{ll}
   1,                          &\mbox{if}~e\in E(v_3);       \nonumber   \\
   2,                          &\mbox{if}~e\in E(v_4),       \nonumber
   \end{array}\right.
\end{eqnarray}
and every other edges obtains distinct colorings from $\{3,\ldots,17\}$ (see
Fig.~5$(a)$). Note that each $C_6\in \Psi$ contains two edges of $E(v_i)$ ($i\in
\{3,4\}$), and so $P_{7,2}$ has no rainbow $C_6$. Hence, $Ar(P_{7,2}, C_6)\geq
17$.

Next we will show that $Ar(P_{7,2}, C_6)\leq 17$. Let
$V(H_{P_{7,2},\Psi})=\{x_i~|~0\leq i\leq 6\}$ and
$E(H_{P_{7,2},\Psi})=\{F_0,\ldots,F_{20}\}$, where $x_i$ corresponds to $G_i$ in $\Psi$, each $F_i$ corresponds to $u_iu_{i+1}$,
$F_{i+7}$ corresponds to $u_iv_i$ and  $F_{i+14}$ corresponds to $v_iv_{i+2}$ for $0\leq i\leq 6$. Note that each $u_iu_{i+1}$ is contained in $G_i$, $u_iv_i$ is
contained in $G_i$ and $G_{i+6}$, $v_iv_{i+2}$ is contained in $G_{i+2}$,
$G_{i+4}$ and $G_{i+6}$ for $0\leq i\leq 6$. It follows that $F_i=\{x_i\}$,
$F_{i+6}=\{x_i,x_{i+6}\}$ and $F_{i+12}=\{x_{i+2},x_{i+4},x_{i+6}\}$ for $0\leq
i\leq 6$, and thus $r(H_{P_{7,2},\Psi})=3$ and $s(H_{P_{7,2},\Psi})=2$. By Lemma \ref{2.3}, $H_{P_{7,2},\Psi}\in \mathscr P_{3,7}$, and then by Lemma \ref{2.4}, $Ar(P_{7,2}, C_6)=Ar_{P_{7,2}}(\Psi)\leq
|E(P_{7,2})|-(3+1)=17$.
\end{proof}

\begin{figure}[!htb]
\centering
{\includegraphics[height=0.35\textwidth]{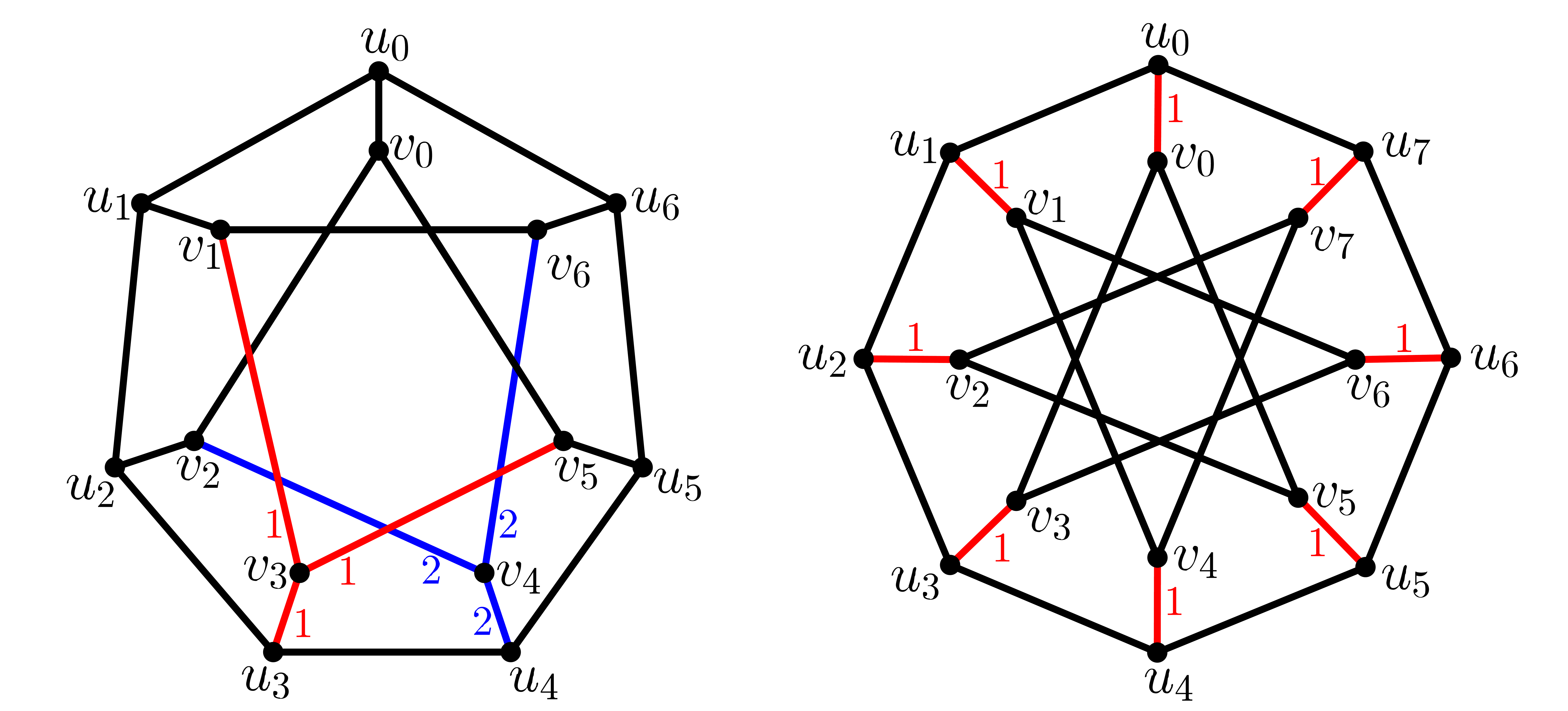}}

\vskip .1cm

$(a)$~$P_{7,2}$~~~~~~~~~~~~~~~~~~~~~~~~~~~~~~~~~~~~~~$(b)$~$P_{8,3}$

\vskip .2cm

Fig.~5 ~$P_{7,2}$ and $P_{8,3}$ have no rainbow $C_6$ under the coloring $\phi$.

\end{figure}

\begin{lemma}
\label{3.144}
$Ar(P_{12,2}, C_6)=34$.
\end{lemma}


\begin{proof} Let $P'_{12,2}=P_{12,2}-E'$, where $E'=\{u_iu_{i+1}, u_iv_i~|~0\leq
i\leq 11\}$, and let $G_i:=v_iv_{i+2}v_{i+4}v_{i+6}v_{i+8}v_{i+10}v_i$ for $0\leq
i\leq 1$. Then $\Psi'=\{G_0,G_1\}$ is the set of all $C_6$ in $P'_{n,k}$, and
then $Ar(P'_{12,2}, C_6)=Ar_{P'_{12,2}}(\Psi')$. Note that each edge in
$P'_{12,2}$ is only contained in some $G_{i}$. Thus $H_{P'_{12,2},\Psi'}$ is a
graph obtained from two isolated vertices by adding six loops at each vertex. Then $M(H_{P'_{12,2},\Psi'})=0$. By Lemma \ref{2.1}, $Ar(P'_{12,2},
C_6)=Ar_{P'_{12,2}}(\Psi')=|E(P'_{12,2})|-|\Psi'|+M(H_{P'_{12,2},\Psi'})=10$.
Since there is no $C_6$ in $P_{12,2}$ containing any edge of $E'$,
$Ar(P_{n,k},C_6)=|E'|+Ar(P'_{n,k},C_6)=34$.
\end{proof}

By Lemmas \ref{3.12}-\ref{3.144}, the proof of Theorem \ref{1.4} is completed.


\subsection{$k=3$}

In this subsection, we focus on the case that $k=3$. Then $n\ge 7$.

\begin{lemma}
\label{3.15}
$Ar(P_{8,3}, C_6)=17$.
\end{lemma}


\begin{proof} Let $G^1_i:=u_iu_{i+1}u_{i+2}u_{i+3}v_{i+3}v_iu_i$,
$G^2_i:=u_iu_{i+1}v_{i+1}v_{i+4}v_{i+7}u_{i+7}u_i$ and
$G^3_i:=u_iu_{i+1}v_{i+1}v_{i+6}v_{i+3}v_iu_i$ for $0\leq i\leq 7$. Then
$\Psi=\{G^1_i,G^2_i,G^3_i~|~0\leq i\leq 7\}$ is the set of all copies of $C_6$ in
$P_{8,3}$, and then $Ar(P_{8,3}, C_6)=Ar_{P_{8,3}}(\Psi)$.

Let $\phi: E(P_{8,3})\rightarrow \{1,2,\ldots,17\}$ be an edge-coloring of
$P_{8,3}$ such that $\phi(u_iv_i)=1$ for $0\leq i\leq 7$ and every other edges
having distinct colorings from $\{2,\ldots,17\}$ (see Fig.~5$(b)$). Note that
each $C_6\in \Psi$ contains two spokes, and hence $P_{8,3}$ has no rainbow $C_6$
under the coloring $\phi$. So, $Ar(P_{8,3}, C_6)\geq 17$.

Next we will show that $Ar(P_{8,3}, C_6)\leq 17$. Let
$V(H_{P_{8,3},\Psi})=\{x_i,y_i,z_i~|~0\leq i\leq 7\}$ and
$E(H_{P_{8,3},\Psi})=\{F^0_i,F^1_i,F^2_i~|~0\leq i\leq 7\}$, where $x_i$, $y_i$
and $z_i$ correspond to $G^1_i$, $G^2_i$ and $G^3_i$ in $\Psi$,
respectively, each $F^1_i$ corresponds to $u_iu_{i+1}$, $F^2_i$
corresponds to $u_iv_i$, and $F^3_i$ corresponds to
$v_iv_{i+3}$ for $0\leq i\leq 7$. Note that each edge $u_iu_{i+1}$ is contained in $G^1_i$, $G^1_{i+6}$,
$G^1_{i+7}$, $G^2_i$, $G^2_{i+1}$ and $G^3_i$; each spoke $u_iv_i$ is contained
in $G^1_i$, $G^1_{i+5}$, $G^2_{i+1}$, $G^2_{i+7}$, $G^3_{i}$ and $G^3_{i+7}$;
each edge $v_iv_{i+3}$ is contained in $G^1_i$, $G^2_{i+4}$, $G^2_{i+7}$,
$G^3_{i}$, $G^3_{i+2}$ and $G^3_{i+5}$ for $0\leq i\leq 7$. It follows that
$F^1_i=\{x_i,x_{i+6},x_{i+7},y_i,y_{i+1},z_i\}$,
$F^2_i=\{x_i,x_{i+5},y_{i+1},y_{i+7},z_i,z_{i+7}\}$ and
$F^3_i=\{x_i,y_{i+4},y_{i+7},z_i,z_{i+2},z_{i+5}\}$ for $0\leq i\leq 7$, and thus
$r(H_{P_{8,3},\Psi})=6$, $s(H_{P_{8,3},\Psi})=3$.

Now we will show that $H_{P_{8,3},\Psi}\in \mathscr{P}_{6,24}$. Suppose that $H_{P_{8,3},\Psi}\notin \mathscr{P}_{6,24}$. Then there exists a $18$-partition $\{\mathscr{F}_1,\mathscr{F}_2,\ldots,\mathscr{F}_{18}\}$ of
$E(H_{P_{8,3},\Psi})$ with $|\bigcup_{i=1}^{18}L_i|=24$. Assume that $|\mathscr{F}_1|\geq  \cdots \geq
|\mathscr{F}_{18}|\ge 1$. Then $2\leq |\mathscr{F}_1|\leq 7$ and
$\sum_{i=1}^{18}(|\mathscr{F}_i|-1)=6$.

Define a graph $Q=(V(Q), E(Q))$ (see Fig.~6) as: $V(Q)=\{F^1_i,F^2_i,F^3_i~|~0\leq
i\leq 7\}$,  and for $x,y\in V(Q)$, $xy\in E(Q)$ if and only if $|x\cap y|=3$.
The subgraph of $Q$ induced by the elements of
$\mathscr{F}_i$ is denoted by $Q_i$. Note that if $\mathscr{F}_i$ is a barreier,
then $Q_i$ is a $2$-factor of order $|\mathscr{F}_i|$.


\begin{figure}[!htb]
\centering
{\includegraphics[height=0.4\textwidth]{6}}

\vskip.1cm

Fig.~6 ~$Q$

\end{figure}

If $|\mathscr{F}_2|\leq 2$, then by $(1)$, $|\bigcup_{i=1}^{18}L_i|\le \sum_{i=1}^{18}|L_i|\leq
3|\mathscr{F}_1|+\sum_{i=2}^{18}3(|\mathscr{F}_i|-1) \leq
3\sum_{i=1}^{18}(|\mathscr{F}_i|-1)+3=21<24$, a contradiction.
So $|\mathscr{F}_2|\geq 3$, and then $3\le |\mathscr{F}_1|\leq 5$.

If $|\mathscr{F}_3|\leq 2$, then $6\le |\mathscr{F}_1|+|\mathscr{F}_2|\le 8$, and by $(1)$, $$24=\left|\bigcup_{i=1}^{18}L_i\right|\le \sum_{i=1}^{18}|L_i|\leq
3|\mathscr{F}_1|+3|\mathscr{F}_2|+\sum_{i=3}^{18}3(|\mathscr{F}_i|-1) =
3\sum_{i=1}^{18}(|\mathscr{F}_i|-1)+6=24,\eqno(2)$$ which implies
$|\mathscr{F}_1|=5$, $|\mathscr{F}_2|=3$ and $|\mathscr{F}_3|=1$, or
$|\mathscr{F}_1|=|\mathscr{F}_2|=4$ and $|\mathscr{F}_3|=1$, or $|\mathscr{F}_1|=4$, $|\mathscr{F}_2|=3$, $|\mathscr{F}_3|=2$ and $|\mathscr{F}_4|=1$, or
$|\mathscr{F}_1|=|\mathscr{F}_2|=3$, $|\mathscr{F}_3|=|\mathscr{F}_4|=2$ and
$|\mathscr{F}_5|=1$. Moreover, $\{L_1,L_2,\ldots, L_{18}\}$ is a partition of
$V(H_{P_{8,3},\Psi})$, $\mathscr{F}_1$, $\mathscr{F}_2$ are barriers. Note that
$Q$ contains no $K_4$ and $K_5$, and hence $|\mathscr{F}_1|=|\mathscr{F}_2|=3$,
$|\mathscr{F}_3|=|\mathscr{F}_4|=2$ and $|\mathscr{F}_5|=1$. By symmetry, we can
assume $\mathscr{F}_1=\{F^1_0,F^1_{7},F^2_0\}$. Then $L_1=\{x_0,x_5,x_6,x_7,y_0,y_1,y_7,z_0,z_7\}$, and then
$\mathscr{F}_2=\{F^1_3,F^1_4,F^2_4\}$ as $L_1\cap L_2=\emptyset$. So
$L_2=\{x_1,x_2,x_3,x_4,y_3,y_4,y_5,z_3,z_4\}$ and $L_3\cup
L_4=\{y_2,y_6,z_1,z_2,z_5,z_6\}$.
On the other hand, by (2), $|L_3|=|L_4|=3$. Since $x_i\notin L_3\cup
L_4$ for $0\le i\le 7$, $F_i^1F_{i+1}^1$,  $F_i^1F_i^2$, $F_i^1F_{i+1}^2$ and
$F_i^2F_{i}^3$ would be excluded from $E(Q)$. Similarly, each edge on two paths
$F_0^3F_3^3F_6^3F_1^3$, $F_4^3F_7^3F_2^3F_5^3$ would be excluded from $E(Q)$ as
$z_0,z_3,z_4,z_7\notin L_3\cup L_4$. So $\mathscr{F}_3, \mathscr{F}_4\in \{\{F_0^3,F_5^3\}, \{F_1^3,F_4^3\}\}$, and thus $y_2,y_6\notin L_3\cup L_4$, a contradiction.

Thus $|\mathscr{F}_3|\geq 3$, and so
$|\mathscr{F}_1|=|\mathscr{F}_2|=|\mathscr{F}_3|=3$ and $|\mathscr{F}_4|=1$. Then
by (1), $|L_i|\le 9$ for $1\le i\le 3$. If $Q_i$ is not a $K_3$ for some $1\le
i\le 3$, then we can check that $|L_i|\le 6$.  So we can assume $|L_1|=|L_2|=9$,
and then $Q_1,~Q_2$ are $K_3$.
If $Q_3$ is not a $K_3$, then $|L_3|=6$ and $F_i$ are barriers as $|L_1\cup L_2\cup L_3|=24$ for $i=1,2,3$,
and thus by an argument similar to the above, we can have a contradiction.
Therefore $Q_3=K_3$ and $|L_3|=9$. If $L_i\cap L_j=\emptyset$ for some
$1\le i<j\le 3$, then by an argument similar to the above, we
can have a contradiction. So $L_i\cap L_j\neq \emptyset$ for any $1\le i<j\le 3$.
Assume that $\mathscr{F}_1=\{F^1_0,F^1_{7},F^2_0\}$. Then
$L_1=\{x_0,x_5,x_6,x_7,y_0,y_1,y_7,z_0,z_7\}$, and then $\mathscr{F}_i\neq
\{F^1_3,F^1_4,F^2_4\}$ as $L_1\cap L_i=\emptyset$ for $i=2,3$. Furthermore, we can check that $|L_i\cap
L_1|=3$ for $i=2,3$. Note that $Q$ is 4-regular, and hence there exists some $x\in (L_1\cap L_2)\setminus (L_1\cap L_3)$. Thus
$|\bigcup_{i=1}^{18}L_i|\le 9+9+9-3-1=23<24$, a contradiction. Therefore, $H_{P_{8,3},\Psi}\in
\mathscr{P}_{6,24}$. By Lemma \ref{2.4}, $Ar(P_{8,3},C_6)=Ar_{P_{8,3}}(\Psi)\leq
|E(P_{8,3})|-(6+1)=17$.
\end{proof}


\begin{figure}[!htb]
\centering
{\includegraphics[height=0.35\textwidth]{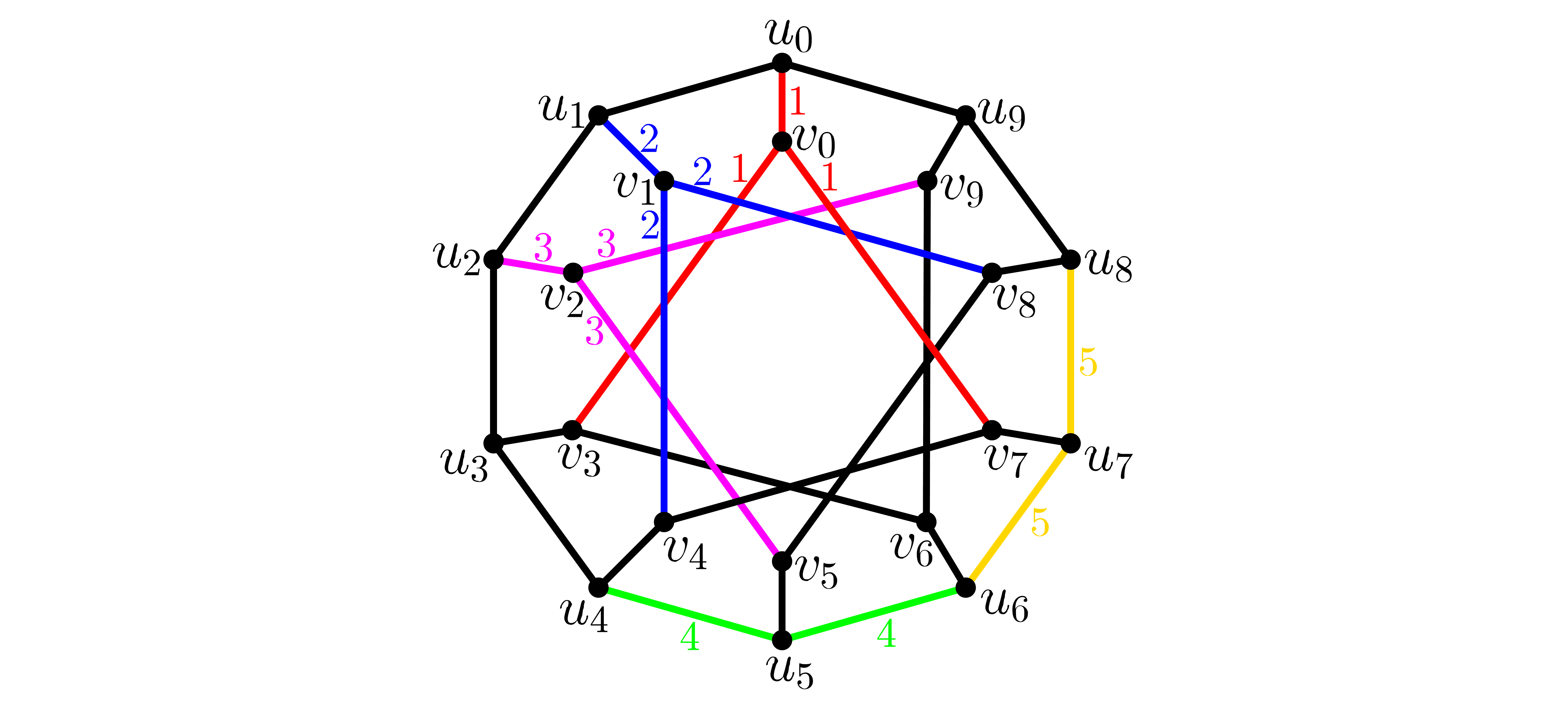}}

Fig.~7 ~$P_{10,3}$ has no rainbow $C_6$ under the coloring $\phi$.

\end{figure}


\begin{lemma}
\label{3.16}
$Ar(P_{10,3}, C_6)=22$.
\end{lemma}


\begin{proof} Note that $\Psi=\{ G^1_i:=u_iu_{i+1}u_{i+2}u_{i+3}v_{i+3}v_iu_i,~
G^2_i:=u_iu_{i+1}v_{i+1}v_{i+4}v_{i+7}v_iu_i~|~0\leq i\leq 9)\}$ is the set of
all copies of $C_6$ in $P_{10,3}$. Then $Ar(P_{10,3}, C_6)=Ar_{P_{10,3}}(\Psi)$.
Let $\phi: E(P_{10,3})\rightarrow \{1,2,\ldots,22\}$ be an edge-coloring of
$P_{10,3}$ defined as follows:
\begin{eqnarray}
   \phi(e)=\left\{
\begin{array}{ll}
   i+1,                          &\mbox{if}~e\in E(v_i)~\mbox{with}~0\leq i\leq
   2;       \nonumber   \\
   4,                            &\mbox{if}~e\in E(u_5)\setminus \{u_5v_5\};
   \nonumber   \\
   5,                            &\mbox{if}~e\in E(u_7)\setminus \{u_7v_7\};
   \nonumber
   \end{array}\right.
\end{eqnarray}
and every other edges obtains distinct colorings from $\{6,\ldots,22\}$ (see
Fig.~7). Note that each $C_6\in \Psi$ contains two edges of $E(v_i)$ ($0\leq
i\leq 2$) or two edges of $E(u_j)\setminus \{u_jv_j\}$ ($j\in \{5,7\}$), and
hence $P_{10,3}$ has no rainbow $C_6$ under the coloring $\phi$. So,
$Ar(P_{10,3}, C_6)\geq 22$.

Next, we will show that $Ar(P_{10,3}, C_6)\leq 22$. Let
$V(H_{P_{10,3},\Psi})=\{x_i,y_i~|~0\leq i\leq 9\}$ and
$E(H_{P_{10,3},\Psi})=\{F^1_i,F^2_i,F^3_i~|~0\leq i\leq 9\}$, where $x_i$ and
$y_i$ correspond to $G^1_i$ and $G^2_i$ in $\Psi$, respectively, each
$F^1_i$ corresponds to $u_iu_{i+1}$, $F^2_i$ to $u_iv_i$, and $F^3_i$ to $v_iv_{i+3}$ for $0\leq
i\leq 9$. Note that each edge $u_iu_{i+1}$ is contained in $G^1_i$, $G^1_{i+8}$,
$G^1_{i+9}$ and $G^2_i$, $u_iv_i$ is contained in $G^1_i$,
$G^1_{i+7}$, $G^2_{i}$ and $G^2_{i+9}$, $v_iv_{i+3}$ is contained in
$G^1_i$, $G^2_{i+3}$, $G^2_{i+6}$ and $G^2_{i+9}$ for $0\leq i\leq 9$. It follows that $F^1_i=\{x_i,x_{i+8},x_{i+9},y_i\}$,
$F^2_i=\{x_i,x_{i+7},y_i,y_{i+9}\}$ and $F^3_i=\{x_i,y_{i+3},y_{i+6},y_{i+9}\}$
for $0\leq i\leq 9$, and thus $r(H_{P_{10,3},\Psi})=4$, $s(H_{P_{10,3},\Psi})=2$

Now we will show that $H_{P_{10,3},\Psi}\in \mathscr{P}_{7,20}$. Suppose that $H_{P_{10,3},\Psi}\notin \mathscr{P}_{7,20}$. Then there exists a $23$-partition $\{\mathscr{F}_1,\mathscr{F}_2,\ldots,\mathscr{F}_{23}\}$ of
$E(H_{P_{10,3},\Psi})$ with  $|\bigcup_{i=1}^{23}L_i|=20$. Assume that $|\mathscr{F}_1|\geq \cdots \geq
|\mathscr{F}_{23}|\ge 1$. Then $2\leq |\mathscr{F}_1|\leq 8$ and
$\sum_{i=1}^{23}(|\mathscr{F}_i|-1)=7$.

Define a graph $Q'=(V(Q'), E(Q'))$ as: $V(Q')=\{F^1_i,F^2_i,F^3_i~|~0\leq
i\leq 9\}$, and for any $x,y\in V(Q')$, $xy\in E(Q')$ if and only if $|x\cap
y|=2$. Then $Q'$ is 4-regular and contains no $K_4$ (see Fig.~8). The subgraph of
$Q'$ induced by the elements of $\mathscr{F}_i$ is denoted by $Q'_i$.

If $|\mathscr{F}_3|\leq 2$, then by $(1)$, $\sum_{i=1}^{23}|L_i|\leq
2|\mathscr{F}_1|+2|\mathscr{F}_2|+\sum_{i=3}^{23}2(|\mathscr{F}_i|-1) \leq
2\sum_{i=1}^{23}(|\mathscr{F}_i|-1)+4=18<20$, a contradiction. So
$|\mathscr{F}_3|\geq 3$. Then $3\leq |\mathscr{F}_1|\leq 4$,
$|\mathscr{F}_2|=|\mathscr{F}_3|=3$ and $|\mathscr{F}_4|\leq 2$. By
$(1)$, $$20= |\bigcup_{i=1}^{23}L_i|\le \sum_{i=1}^{23}|L_i|\leq
2|\mathscr{F}_1|+2|\mathscr{F}_2|+2|\mathscr{F}_3|+\sum_{i=4}^{23}2(|\mathscr{F}_i|-1)
\leq 2\sum_{i=1}^{23}(|\mathscr{F}_i|-1)+6=20,$$ which implies
$\mathscr{F}_1,\mathscr{F}_2,\mathscr{F}_3$ are barriers and $\{L_1,L_2,\ldots,
L_{23}\}$ is a partition of $V(H_{P_{10,3},\Psi})$.

Since $Q'$ contains no $K_4$, $|\mathscr{F}_1|\neq 4$. So
$|\mathscr{F}_1|=|\mathscr{F}_2|=|\mathscr{F}_3|=3$ and $|\mathscr{F}_4|=2$. Then
each $Q'_i$ is a $K_3$ for $1\le i\le 3$. Without loss of generality, assume
$\mathscr{F}_1=\{F^1_0,F^1_9,F^2_0\}$. Then $L_1=\{x_0,x_7,x_8,x_9,y_0,y_9\}$.
Since $L_i\cap L_j=\emptyset$ for $1\le i<j\le 3$, we have
$\{F^1_3,F^1_4,F^1_5,F^1_6,F^2_4,F^2_5,F^2_6,F^3_2,F^3_5\}\subseteq
\mathscr{F}_2\cup \mathscr{F}_3$. But in any case, we can check $L_2\cap L_3\neq
\emptyset$, a contradiction. Therefore, $H_{P_{10,3},\Psi}\in
\mathscr{P}_{7,20}$. By Lemma \ref{2.4}, $Ar(P_{10,3},
C_6)=Ar_{P_{10,3}}(\Psi)\leq |E(P_{10,3})|-(7+1)=22$.
\end{proof}

\begin{figure}[!htb]
\centering
{\includegraphics[height=0.4\textwidth]{8}}

\vskip.1cm

Fig.~8 ~$Q'$

\end{figure}


\begin{lemma}
\label{3.17}
$Ar(P_{18,3}, C_6)=42$.
\end{lemma}


\begin{proof} Note that $\Psi=\{ G^1_i, G^2_j~|~0\leq i\leq 17,~0\leq j\leq 2\}$
is the set of all copies of $C_6$ in $P_{18,3}$, where $
G^1_i:=u_iu_{i+1}u_{i+2}u_{i+3}v_{i+3}v_iu_i,
G^2_j:=v_jv_{j+3}v_{j+6}v_{j+9}v_{j+12}v_{j+15}v_j$. Then $Ar(P_{18,3},
C_6)=Ar_{P_{18,3}}(\Psi)$.

Let $\phi: E(P_{18,3})\rightarrow \{1,2, \ldots, 42\}$ be an edge-coloring of
$P_{18,3}$ defined as follows:
\begin{eqnarray}
   \phi(e)=\left\{
\begin{array}{ll}
   i+1,                             &\mbox{if}~e\in
   \{u_{2i}u_{2i+1},u_{2i+1}u_{2i+2}\},~0\leq i\leq 8;              \nonumber
   \\
   j+10,                            &\mbox{if}~e\in
   \{v_{j+3}v_j,v_jv_{j-3}\},~0\leq j\leq 2;
   \nonumber
   \end{array}\right.
\end{eqnarray}
and every other edges obtains distinct colorings from $\{13,\ldots,42\}$ (see
Fig.~9). Note that each $C_6\in \Psi$ contains two edges
$u_{2i}u_{2i+1},u_{2i+1}u_{2i+2}$ for some $0\leq i\leq 8$, or two edges
$v_{j+3}v_j,v_jv_{j-3}$ for some $0\leq j\leq 2$, and hence $P_{18,3}$ has no
rainbow $C_6$ under the coloring $\phi$. So, $Ar(P_{18,3}, C_6)\geq 42$.

Next, we will show that $Ar(P_{18,3}, C_6)\leq 42$. Let $V(H_{P_{18,3},\Psi})=\{x_i,y_j~|~0\leq i\leq 17,~0\leq j\leq 2\}$ and
$E(H_{P_{18,3}),\Psi})=\{F^1_i,F^2_i,F^3_i~|~0\leq i\leq 17\}$, where $x_i$ and
$y_j$ correspond to $G^1_i$ and $G^2_j$ in $\Psi$, respectively, each $F^1_i$
corresponds to $u_iu_{i+1}$, $F^2_i$ to  $u_iv_i$, and $F^3_i$ to $v_iv_{i+3}$ for $0\leq i\leq
17$. Note that each edge $u_iu_{i+1}$ is contained in $G^1_i$, $G^1_{i+16}$ and
$G^1_{i+17}$, $u_iv_i$ is contained in $G^1_i$ and $G^1_{i+15}$, $v_{3i+t}v_{3i+t+3}$ is contained in $G^1_{3i+t}$ and $G^2_t$ for $0\leq
i\leq 17$ and $t=\lfloor \frac{i}{6} \rfloor$. Hence $F^1_i=\{x_i,x_{i+16},x_{i+17}\}$, $F^2_i=\{x_i,x_{i+15}\}$
and $F^3_i=\{x_{3i+t},y_{t}\}$ for $0\leq i\leq 17$ and $t=\lfloor \frac{i}{6} \rfloor$, and thus
$r(H_{P_{18,3},\Psi})=3$, $s(H_{P_{18,3},\Psi})=2$.

Now we will show that $H_{P_{18,3},\Psi}\in \mathscr{P}_{11,21}$. Suppose that $H_{P_{18,3},\Psi}\notin \mathscr{P}_{11,21}$. Then there exists a $43$-partition
$\{\mathscr{F}_1,\mathscr{F}_2,\ldots,\mathscr{F}_{43}\}$ of
$E(H_{P_{18,3},\Psi})$ with $
|\bigcup_{i=1}^{43}L_i|=21$, and then $
\bigcup_{i=1}^{43}L_i=V(H_{P_{18,3},\Psi})$. Note that $\sum_{i=1}^{43}(|\mathscr{F}_i|-1)=11$. By $(1)$, $|L_i|\leq 2(|\mathscr{F}_i|-1)$ for $1\le i\le
43$.
 Assume $y_j\in L_{i_j}$ for $j=0,1,2$. Note that $s(H_{P_{18,3},\Psi})=2$ and $|F^3_i\cap F^3_j|\le 1$ for $0\le i<j\le 17$,
and hence, by (1) \begin{eqnarray}
   |L_{i_j}|\leq \left\{
\begin{array}{ll}
   \lfloor \frac{3|\mathscr{F}_{i_j}|}{2} \rfloor-1\le
   2(|\mathscr{F}_{i_j}|-1)-1,               &\mbox{if}~|\mathscr{F}_{i_j}|\geq
   3;           \nonumber    \\
   1,
   &\mbox{if}~|\mathscr{F}_{i_j}|=2.               \nonumber
   \end{array}\right.
\end{eqnarray} Moreover, if $i_t=i_j$ for some $0\le t<j\le 2$, then
$|\mathscr{F}_{i_j}|\geq 4$ and $ |L_{i_j}|\leq  2(|\mathscr{F}_{i_j}|-1)-2$.
So $\sum_{i=1}^{43}|L_i|\leq 2\sum_{i=1}^{43}(|\mathscr{F}_i|-1)-2=20<21$, a
contradiction. Therefore, $H_{P_{18,3},\Psi}\in \mathscr{P}_{11,21}$. By Lemma
\ref{2.4}, $Ar(P_{18,3}, C_6)=Ar_{P_{18,3}}(\Psi)\leq |E(P_{18,3})|-(11+1)=42$.
\end{proof}


\begin{figure}[!htb]
\centering
{\includegraphics[height=0.5\textwidth]{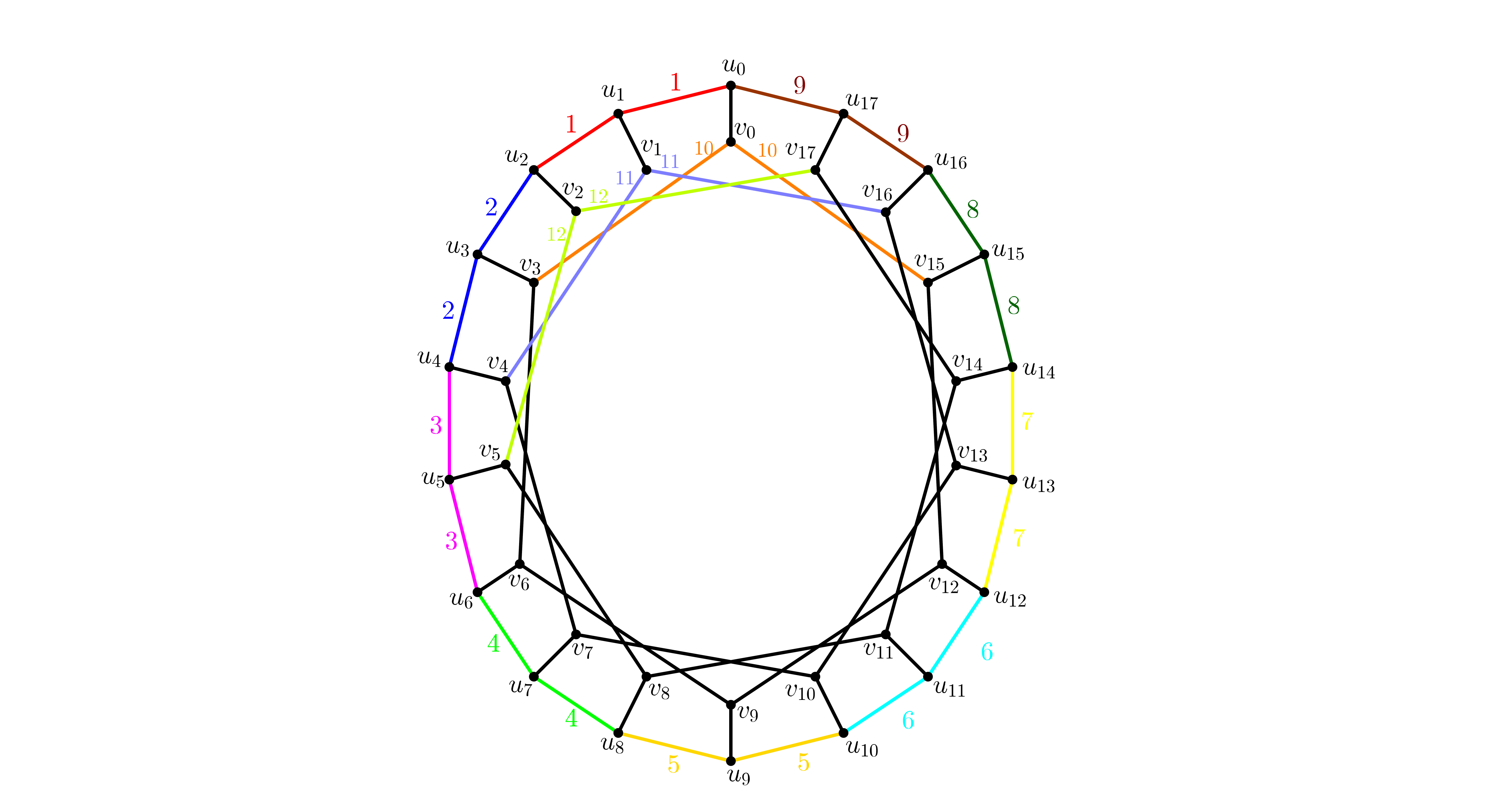}}

Fig.~9 ~$P_{18,3}$ has no rainbow $C_6$ under the coloring $\phi$.

\end{figure}


\begin{lemma}
\label{3.18}
$Ar(P_{n,3}, C_6)=\lfloor \frac{5n}{2} \rfloor$ for $n\geq 7$ and $n\notin
\{8,10,18\}$.
\end{lemma}


\begin{proof} Let $G_i:=u_iu_{i+1}u_{i+2}u_{i+3}v_{i+3}v_iu_i$ for $0\leq i\leq
n-1$. Then $\Psi=\{ G_i~|~0\leq i\leq n-1\}$ is the set of all copies of $C_6$ in
$P_{n,3}$, and then $Ar(P_{n,3}, C_6)=Ar_{P_{n,3}}(\Psi)$.

Let $\phi: E(P_{n,3})\rightarrow \{1,2, \ldots, \lfloor \frac{5n}{2} \rfloor\}$
be an edge-coloring of $P_{n,3}$ defined as follows:
\begin{eqnarray}
   \phi(e)=\left\{
\begin{array}{ll}
   i+1,                             &\mbox{if}~e\in
   \{u_{2i}u_{2i+1},u_{2i+1}u_{2i+2}\},~0\leq i\leq \lfloor \frac{n}{2}
   \rfloor-1;     \nonumber   \\
   \lfloor \frac{n}{2} \rfloor,     &\mbox{if}~e=u_0u_{n-1}~\mbox{and}~n\equiv
   1\pmod 2,
   \nonumber
   \end{array}\right.
\end{eqnarray}
and every other edges obtains distinct colorings from $\{\lfloor \frac{n}{2}
\rfloor+1,\ldots,\lfloor \frac{5n}{2} \rfloor\}$. Note that if $n\equiv 0\pmod
2$, then each $C_6\in \Psi$ contains two edges $u_{2i}u_{2i+1},u_{2i+1}u_{2i+2}$
for some $0\leq i\leq \lfloor \frac{n}{2} \rfloor-1$; and if $n\equiv 1\pmod 2$,
then each $C_6\in \Psi$ contains two edges $u_{2i}u_{2i+1},u_{2i+1}u_{2i+2}$ for
some $0\leq i\leq \lfloor \frac{n}{2} \rfloor-2$ or two of
$\{u_{n-3}u_{n-2},u_{n-2}u_{n-1},u_0u_{n-1}\}$, and hence $P_{n,3}$ has no
rainbow $C_6$ under the coloring $\phi$. So, $Ar(P_{n,3}, C_6)\geq \lfloor
\frac{5n}{2} \rfloor$.

Next, we will show that $Ar(P_{n,3}, C_6)\leq \lfloor
\frac{5n}{2} \rfloor$. Note that each edge $u_iu_{i+1}$ is contained in $G_i$, $G_{i+n-2}$ and
$G_{i+n-1}$, $u_iv_i$ is contained in $G_i$ and $G_{i+n-3}$,
$v_iv_{i+3}$ is contained in $G_i$ for $0\leq i\leq n-1$. Then
$V(H_{P_{n,3},\Psi})=\{x_i~|~0\leq i\leq n-1\}$ and
$E(H_{P_{n,3},\Psi})=\{F_0,\ldots,F_{3n-1}\}$ with
$F_i=\{x_i,x_{i+n-2},x_{i+n-1}\}$, $F_{i+n}=\{x_i,x_{i+n-3}\}$ and
$F_{i+2n}=\{x_i\}$ for $0\leq i\leq n-1$. Thus $r(H_{P_{n,3},\Psi})=3$ and
$s(H_{P_{n,3},\Psi})=1$. By Lemma \ref{2.2}, we have $Ar(P_{n,3},
C_6)=Ar_{P_{n,3}}(\Psi)\leq |E(P_{n,3})|-\lceil \frac{2|\Psi|}{3+1}
\rceil=3n-\lceil \frac{n}{2} \rceil=\lfloor \frac{5n}{2} \rfloor$.
\end{proof}

By Lemmas \ref{3.15}-\ref{3.18}, the proof of Theorem \ref{1.5} is completed.


\subsection{$k\ge 4$}

In this subsection, $k\in \{\frac{n}{6},\frac{n-1}{3},\frac{n+1}{3}\}$ by
Proposition \ref{1.1}(ii).

\begin{lemma}
\label{3.19}
$Ar(P_{n,k}, C_6)=\frac{17n}{6}$ for $k=\frac{n}{6}\ge 4$.
\end{lemma}


\begin{proof}  Denote $P'_{n,k}=P_{n,k}-E'$, where $E'=\{u_iu_{i+1}, u_iv_i~|~0\leq
i\leq \frac{n}{6}-1\}$. Let $G_i:=v_iv_{i+k}v_{i+2k}v_{i+3k}v_{i+4k}v_{i+5k}v_i$
for $0\leq i\leq \frac{n}{6}-1$. Then $\Psi'=\{G_i~|~0\leq i\leq \frac{n}{6}-1\}$
is the set of all copies of $C_6$ in $P'_{n,k}$, and $Ar(P'_{n,k},
C_6)=Ar_{P'_{n,k}}(\Psi')$. Note that each edge in $P'_{n,k}$ is only contained
in some $G_{i}$. Thus $H_{P'_{n,k},\Psi'}$ is a graph obtained from $\frac{n}{6}$
isolated vertices by adding six loops on each isolated vertex. Then
$M(H_{P'_{n,k},\Psi'})=0$. By Lemma \ref{2.1}, we have $Ar(P'_{n,k},
C_6)=Ar_{P'_{n,k}}(\Psi')=|E(P'_{n,k})|-|\Psi'|+M(H_{P'_{n,k},\Psi'})=\frac{5n}{6}$
for $k=\frac{n}{6}\ge 4$. Since there is no copy of $C_6$ in $P_{n,k}$ containing
any edge of $F$, $Ar(P_{n,k},C_6)=|E'|+Ar(P'_{n,k},C_6)=\frac{17n}{6}$.
\end{proof}

\begin{figure}[!htb]
\centering
{\includegraphics[height=0.35\textwidth]{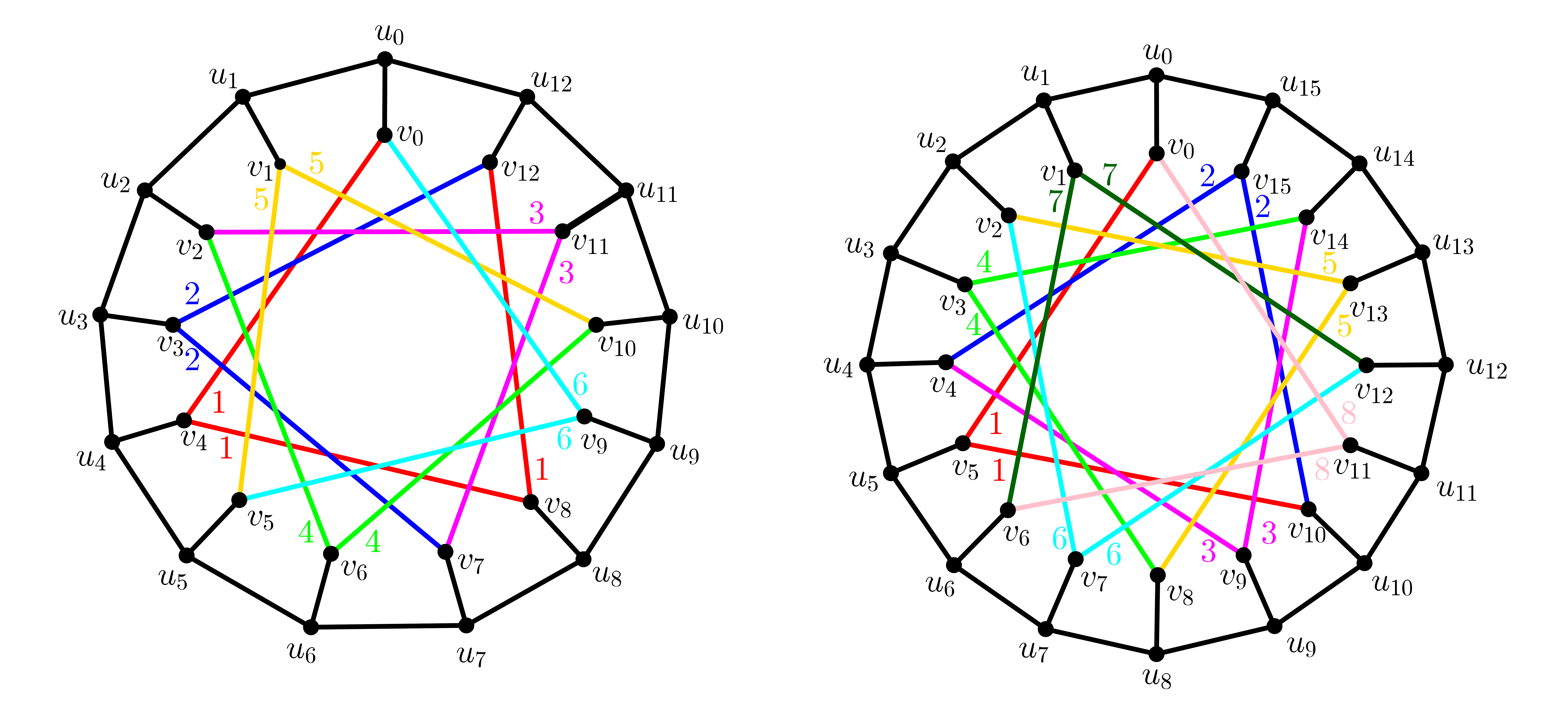}}

$(a)$~$P_{13,4}$~~~~~~~~~~~~~~~~~~~~~~~~~~~~~~~~~~~~~~$(b)$~$P_{16,5}$

\vskip .1cm

Fig.~10 ~Illustrations of $P_{13,4}$ and $P_{16,5}$ having no rainbow $C_6$ under
$\phi$.

\end{figure}

\begin{lemma}
\label{3.20}
$Ar(P_{n,k}, C_6)=\lfloor \frac{5n}{2} \rfloor$ for $k\in
\{\frac{n-1}{3},\frac{n+1}{3}\}$.
\end{lemma}


\begin{proof} Let \begin{eqnarray}
  G_i:=\left\{
\begin{array}{ll}
  u_iu_{i+1}v_{i+1}v_{i+1+k}v_{i+1+2k}v_iu_i,    &\mbox{if}~k=\frac{n-1}{3};
  \\ \nonumber
   u_iv_{i}v_{i+k}v_{i+2k}v_{i+1}u_{i+1}u_i,      &\mbox{if}~k=\frac{n+1}{3},
   \nonumber
   \end{array}\right.
\end{eqnarray}
Then $ \Psi=\{G_i~|~0\leq i\leq n-1\}$ is the set of all copies of $C_6$ in
$P_{n,k}$ for $k\in \{\frac{n-1}{3},\frac{n+1}{3}\}$. Then $Ar(P_{n,k},
C_6)=Ar_{P_{n,k}}(\Psi)$. Recall that $|E(P_{n,k})|=3n$.

If $n\equiv 1\pmod 2$, then let $\phi: E(P_{n,k})\rightarrow \{1,2, \ldots,
\lfloor \frac{5n}{2} \rfloor\}$ be an edge-coloring of $P_{n,k}$ such that
$\phi(v_0v_k)=\phi(v_{k}v_{2k})=\phi(v_{2k}v_{3k})=1$,
$\phi(v_{ik}v_{(i+1)k})=\lceil \frac{i}{2} \rceil$ for all $3\leq i\leq n-1$ and
every other edges obtains distinct colorings from $\{\frac{n+1}{2},\ldots,
\lfloor \frac{5n}{2} \rfloor\}$ (Fig. 10(a) illustrates the coloring $\phi$ of
$P_{13,4}$). Then each $C_6\in \Psi$ contains two edges of $\{v_0v_k,
v_{k}v_{2k}, v_{2k}v_{3k}\}$ or two edges
$v_{(2i+1)k}v_{(2i+2)k},v_{(2i+2)k}v_{(2i+3)k}$ for some $1\leq i\leq
\frac{n-1}{2}-1$, and hence $P_{n,k}$ has no rainbow $C_6$ under the coloring
$\phi$.

If $n\equiv 0\pmod 2$, then let $\phi: E(P_{n,k})\rightarrow \{1,2, \ldots,
\lfloor \frac{5n}{2} \rfloor\}$ be an edge-coloring of $P_{n,k}$ such that
$\phi(v_{ik}v_{(i+1)k})=\lfloor \frac{i}{2} \rfloor+1$ for all $0\leq i\leq n-1$
and every other edges obtains distinct colorings from $\{\frac{n}{2}+1,\ldots,
\lfloor \frac{5n}{2} \rfloor\}$ (Fig. 10(b) illustrates the coloring $\phi$ of
$P_{16,5}$). Note that each $C_6\in \Psi$ contains two edges
$v_{2ik}v_{(2i+1)k},v_{(2i+1)k}v_{(2i+2)k}$ for some $0\leq i\leq \frac{n}{2}-1$,
and hence $P_{n,k}$ has no rainbow $C_6$ under the coloring $\phi$. So,
$Ar(P_{n,k}, C_6)\geq \lfloor \frac{5n}{2} \rfloor$ for $n\equiv 0 \pmod 2$. So,
in either case, $Ar(P_{n,k}, C_6)\geq \lfloor \frac{5n}{2} \rfloor$.

Now, we show that $Ar(P_{n,k}, C_6)\leq \lfloor \frac{5n}{2} \rfloor$. Let
$V(H_{P_{n,k},\Psi})=\{x_i~|~0\leq i\leq n-1\}$ and
$E(H_{P_{n,k},\Psi})=\{F_0,\ldots,F_{3n-1}\}$, where each $x_i$ corresponds to
$G_i$ in $\Psi$, each $F_i$ corresponds to  $u_iu_{i+1}$, $F_{i+n}$
corresponds to $u_iv_i$, $F_{i+2n}$ corresponds to $v_iv_{i+k}$ for $k\in
\{\frac{n-1}{3}, \frac{n+1}{3}\}$. Note that
for $0\leq i\leq n-1$, each $u_iu_{i+1}$ is contained in $G_i$, $u_iv_i$ is
contained in $G_i$, $G_{i+n-1}$, $v_iv_{i+\frac{n-1}{3}}$ is contained in
$G_{i+\frac{n-1}{3}}$, $G_{i+\frac{2n-2}{3}}$, $G_{i+n-1}$ when $k=\frac{n-1}{3}$, and $v_iv_{i+\frac{n+1}{3}}$ is contained in $G_i$, $G_{i+\frac{n-2}{3}}$,
$G_{i+\frac{2n-1}{3}}$ when $k=\frac{n+1}{3}$. It follows that $F_i=\{x_i\}$,
$F_{i+n}=\{x_i,x_{i+n-1}\}$ and
\begin{eqnarray}
   F_{i+2n}=\left\{
\begin{array}{ll}
   \{x_{i+\frac{n-1}{3}},x_{i+\frac{2n-2}{3}}, x_{i+n-1}\},
   &\mbox{if}~k=\frac{n-1}{3};               \\          \nonumber
   \{x_i,x_{i+\frac{n-2}{3}},x_{i+\frac{2n-1}{3}}\},
   &\mbox{if}~k=\frac{n+1}{3}.                           \nonumber
   \end{array}\right.
\end{eqnarray}
Obviously, $r(H_{P_{n,k},\Psi})=3$ and $s(H_{P_{n,k},\Psi})=1$, then by Lemma
\ref{2.2}, $Ar(P_{n,k}, C_6)=Ar_{P_{n,k}}(\Psi)\leq |E(P_{n,k})|-\lceil
\frac{2|\Psi|}{3+1} \rceil=3n-\lceil \frac{n}{2} \rceil=\lfloor \frac{5n}{2}
\rfloor$.
\end{proof}


\begin{lemma}
\label{3.21}
$Ar(P_{n,k}, C_6)=\lfloor \frac{7n}{3} \rfloor$ for $k=\frac{n-2}{2}\ge 4$.
\end{lemma}


\begin{proof} Let $ G_i:=u_iu_{i+1}v_{i+1}v_{i+\frac{n}{2}}v_{i+n-1}u_{i+n-1}u_i$
for $0\leq i\leq n-1$. Then $\Psi=\{ G_i~|~0\leq i\leq n-1\}$ is the set of all
copies of $C_6$ in $P_{n,k}$, and $Ar(P_{n,k}, C_6)=Ar_{P_{n,k}}(\Psi)$. Note
that each edge $u_iu_{i+1}$ is contained in $G_i$ and $G_{i+1}$; each spoke
$u_iv_i$ is contained in $G_{i+1}$ and $G_{i+n-1}$; each edge
$v_iv_{i+\frac{n-2}{2}}$ is contained in $G_{i+\frac{n}{2}}$ and $G_{i+n-1}$ for
$0\leq i\leq n-1$. Thus $H_{P_{n,k},\Psi}$ is a graph obtained from a cycle
$w_0w_1\cdots w_{n-1}w_0$ by adding edges $w_{i+1}w_{i+n-1}$ and
$w_{i+\frac{n}{2}}w_{i+n-1}$ for $0\leq i\leq n-1$. Since there exist $\lfloor
\frac{n}{3} \rfloor$ vertex-disjoint cycles $w_{3i}w_{3i+1}w_{3i+2}w_{3i}$ in
$H_{P_{n,k},\Psi}$ for $0\leq i\leq \lfloor \frac{n}{3} \rfloor-1$, we have
$M(H_{P_{n,k},\Psi})=\lfloor \frac{n}{3} \rfloor$. By Lemma \ref{2.1},
$Ar(P_{n,k},C_6)=Ar_{P_{n,k}}(\Psi)=|E(P_{n,k})|-|\Psi|+M(H_{P_{n,k},\Psi})=\lfloor
\frac{7n}{3} \rfloor$.
\end{proof}

By Lemmas \ref{3.19}-\ref{3.21}, we complete the proof of Theorem \ref{1.6}.



\end{document}